\documentclass[reqno,twoside,11pt]{amsart}
\usepackage{amsbsy,amssymb,amscd,amsfonts,latexsym,amstext,delarray,amsmath,graphicx,color,caption,blkarray,mathtools}
\usepackage{tikz}
\usetikzlibrary{cd}
\usetikzlibrary{matrix,arrows}
\usepackage{graphicx}
\usepackage{hyperref}
\input xy

\xyoption{all}
\usepackage{euscript}
\usepackage{multirow,charter}
\usepackage{etex, pictexwd,dcpic}
\usepackage[normalem]{ulem}
\DeclareMathOperator{\Res}{Res}

\DeclareMathOperator{\reg}{reg}
\newcommand{\tr}{\mathrm{tr}}
\newcommand{\mult}{\mathrm{mult}}
    \advance \topmargin by -\headheight
    \advance \topmargin by -\headsep

    \evensidemargin \oddsidemargin
\newtheorem {theorem}    {Theorem}[section]

\newtheorem {lemma}      [theorem]    {Lemma}
\newtheorem {corollary}  [theorem]    {Corollary}
\newtheorem {proposition}[theorem]    {Proposition}

\newtheorem{conjecture}[theorem]{Conjecture}
\theoremstyle{definition}

\theoremstyle{remark}
\newtheorem{remark}{Remark}

    \advance \topmargin by -\headheight
    \advance \topmargin by -\headsep

    \evensidemargin \oddsidemargin
\DeclareMathOperator{\SL}{SL}
\DeclareMathOperator{\GL}{GL}

\DeclareMathOperator{\Aut}{Aut}

\DeclareMathOperator{\End}{End}

\DeclareMathOperator{\ad}{ad}

\DeclareMathOperator{\re}{{re}}
\DeclareMathOperator{\im}{{im}}

\renewcommand{\a}{\alpha}

\newcommand{\Z}{{\mathbb Z}}
\newcommand{\R}{{\mathbb R}}
\newcommand{\C}{{\mathbb C}}
\newcommand{\N}{{\mathbb N}}
\newcommand{\Q}{{\mathbb Q}}

\usepackage{ragged2e}
\newcommand{\M}{\mathbb{M}}

\author{Darlayne Addabbo}
\email{addabbo@math.arizona.edu}
\author{Lisa Carbone}
\email{lisa.carbone@rutgers.edu}
\author{Elizabeth Jurisich}
\email{jurisiche@cofc.edu}
\author{Maryam Khaqan}
\email{maryam.khaqan@utoronto.ca}
\author{Scott H. Murray}
\email{scotthmurray@gmail.com}

\title{\large{Vertex operators for imaginary $\mathfrak{gl}_2$ subalgebras  of the Monster Lie algebra}}
\date{\today}

\begin{document}

\thanks{The first author is grateful for an AMS-Simons Travel Grant. The second author's research is partially supported by the Simons Foundation, Mathematics and Physical Sciences-Collaboration Grants for Mathematicians, Award Number 422182. }

\begin{abstract}  
The Monster Lie algebra $\mathfrak m$ is a quotient of the physical space of the vertex algebra $V=V^\natural\otimes V_{1,1}$, where $V^\natural$ is the Moonshine module vertex operator algebra of Frenkel, Lepowsky, and Meurman, and $V_{1,1}$ is the vertex algebra corresponding to the rank 2 even unimodular lattice~$\textrm{II}_{1,1}$. %
We construct vertex algebra elements that project to bases for subalgebras of~$\mathfrak m$ isomorphic to~$\mathfrak{gl}_{2}$, corresponding to each imaginary simple root, denoted $(1,j)$ for $j>0$.
Our method requires the existence of pairs of primary vectors in $V^{\natural}$ satisfying some natural conditions, which we prove. We show that the action of the Monster finite simple group $\mathbb{M}$ on the subspace of primary vectors in $V^\natural$ induces an $\M$-action on the set of  $\mathfrak{gl}_2$ subalgebras corresponding to a fixed imaginary simple root. %
We use the generating function for dimensions of subspaces of primary vectors of $V^\natural$ to prove that this action is non-trivial for small values of $j$.
\end{abstract}

\begingroup
\def\uppercasenonmath#1{} %
\let\MakeUppercase\relax %
\maketitle
\endgroup

\section{Introduction}
In 1978, McKay and Thompson \cite{Th} observed that the first few Fourier coefficients of the  normalized\footnote{We use the normalization of \cite{FLMPNAS} and \cite{FLM} where the constant term of $J(q)$ is taken to be zero.} elliptic modular invariant $J(q)$
can be written as sums involving the dimensions of irreducible representations of the Monster finite simple group $\M$ as follows:%
\begin{align*}\label{numerology}
\begin{split}
196884 &= 1 + 196883,\\
21493760 &= 1 + 196883 + 21296876,\\
864299970 &= 2 \cdot 1 + 2 \cdot 196883 + 21296876 + 842609326.
\end{split}
\end{align*}
Recall that
\begin{equation}\label{Jfun}J(q ) = \sum_{n= -1}^\infty c(n)q^n=\dfrac{1}{q} + 196884q +21493760q^2+864299970q^3+\dots,
\end{equation}
where  $q=e^{2\pi{\bf{i}}\tau}$ for $\tau\in\C$ with $\text{Im}(\tau)>0.$ This {observation }gave rise to  the McKay-Thompson conjecture \cite{Th} that there exists a naturally defined graded infinite-dimensional $\M$-module $W = \bigoplus^\infty_{n=-1} W_n,$ which satisfies $\dim(W_{n}) = c(n)$ for $n \in \Z$.%

In the same paper {\cite{Th}}, Thompson also drew attention to the generating functions $$\sum^{\infty}_{n=-1}{\tr}(g|_{W_n} )q^n $$
for $g\in\M$. These became known as  McKay-Thompson series and  were subsequently explicitly described  in a conjecture by Conway and Norton \cite{CN}. %

A defining property of %
$J(q)$ is that it is the normalized {Hauptmodul} for the field of $\SL(2, \Z)$-modular invariant functions on the upper half-plane  (\cite{CN},  \cite {FLM}).   Recall that for a genus-zero curve $X$,  the field of modular-invariant functions on the curve is generated by a single element. If we normalize such a generator so that its constant term is zero and it has leading coefficient 1, then the generating function is unique. Such a  function is then called the normalized {Hauptmodul} 
for~$X$. %

Conway and Norton built upon McKay and Thompson's conjecture to formulate the Monstrous Moonshine Conjecture:
\begin{conjecture}\label{CNconjecture} (\cite{CN}) There exists a natural infinite-dimensional $\Z$-graded  $\M$-module  $W=\bigoplus_{n=-1}^\infty W_n$ such that $\sum_{n=-1}^\infty\dim(W_n)q^n = J(q)$ and such that for  every $g \in \M$, the generating function 
$$\sum_{n=-1}^\infty{\text{\rm tr}}\left(g|_ 
{W_n} \right)q^n $$ 
is  the Hauptmodul of a genus zero function field arising from a suitable discrete subgroup of $\SL(2,\R)$.

\end{conjecture}
 
Note that this conjecture predates the construction of the Monster finite simple group $\M$,  which had been predicted to exist by  Fischer and Griess. The Monster group  $\M$ was constructed in 1982 by Griess \cite{Gr} as a group of automorphisms of a 196883-dimensional commutative nonassociative algebra  over $\Q$, generated by the actions of an involution centralizer $C$ and another involution $\sigma$. (See the Introduction and Remark 10.4.13 of \cite{FLM} for a description of the role of the group $C$ in \cite{Gr} and earlier conjectures).

In their seminal work \cite{FLMPNAS} and \cite{FLM}, Frenkel, Lepowsky and Meurman  constructed a natural infinite-dimensional $\Z$-graded $\M$-module, $V^\natural$, with graded dimension~$J(q)$, thus proving the McKay-Thompson conjecture. 
They constructed the module $V^{\natural}$ as a vertex operator algebra and called it the {\it Moonshine module}.
In particular, if 1 is the identity element of $\M$,  in \cite{FLM} the authors showed that 
$$\tr\left(1|_{V_{n+1}^{\natural}}\right) = \text{dim}\left(V_{n+1}^{\natural}\right)=c(n)$$ and the  McKay-Thompson series $\sum_{n=-1}^\infty\tr\left(1|_{V_{n+1}^{\natural}}\right)$ is $  J(q)$.
In \cite{FLM}, the authors explicitly constructed $\M$ as a group of  vertex operator algebra automorphisms of $V^{\natural}$, by constructing the actions of the involution centralizer $C$ and the extra involution $\sigma$  on $V^{\natural}$.

Note that there is a shift between the grading of $W$ above and of $V^\natural$  in \cite{BoInvent}, as we use the conformal weight grading on $V^\natural$ throughout this paper. That is, let $L(0)$ denote the zero mode of the conformal vector $\omega^\natural$ of $V^\natural$,
  then $V^\natural=\bigoplus_{n=0}^\infty V_{n}^\natural$ where $V_n^\natural=\{v\in V^\natural \mid L(0)v=nv\}$ is the subspace of vectors of weight $n$.

Frenkel, Lepowsky and Meurman's Moonshine module $V^{\natural}$ was constructed as the sum of two subspaces $V_+$ and $V_-$,
which are the $+1$ and $-1$ eigenspaces of a certain involution in $\M$. If
an element $g\in\M$ commutes with this involution, then its  McKay-Thompson series $T^\natural_g(q) :=\sum_{n=-1}^\infty\tr\left(g|_{V_{n+1}^{\natural}}\right)q^n$ may be determined as the sum of two series given
by its traces on~$V_+$ and~$V_-$ (as done in \cite{FLM}). In this way, Frenkel, Lepowsky, and Meurman determined the graded traces for all elements of the involution centralizer $C$,   thus partially settling the Conway-Norton conjecture. For elements of $\M$ that are not conjugate to an element of $C$, it remained to determine their McKay-Thompson series.

In his remarkable works \cite{BoPNAS} and \cite{BoInvent}, Borcherds used the \cite{FLM} construction of $V^\natural$ and his discovery of  a new class of Lie algebras known as generalized Kac-Moody algebras \cite{BoGKMA}, also called Borcherds algebras, together with the No-ghost Theorem from String Theory, to determine the remaining McKay-Thompson series for $V^\natural$. He also proved that these McKay-Thompson series coincide with the modular functions described by Conway and Norton.

This work concerns the Monster Lie algebra ${\mathfrak {m}}$, which was a crucial ingredient of Borcherds's proof \cite{BoInvent}. He constructed ${\mathfrak {m}}$ as a quotient $\mathfrak{m}=P_1/R$ of the {{physical space}} $P_1$ of the vertex algebra $V=V^\natural\otimes V_{1,1}$, where $V^\natural$ is the Moonshine {m}odule {vertex operator algebra} and  $V_{1,1}$ is the vertex algebra for the even unimodular 2-dimensional Lorentzian lattice $\textrm{II}_{1,1}$. Here $R$ is the radical of a natural symmetric bilinear form on $V$.
 
Borcherds used the No-ghost {Theorem} to show that $\mathfrak{m}=P_1/R$ is isomorphic to a Borcherds algebra   \cite{BoInvent}. In \cite{JurJPAA}, an explicit Borcherds Cartan matrix $A$ was given:
\begin{equation}\label{mdec}%
{A=\smaller \begin{blockarray}{cccccccccc}
 & & \xleftrightarrow{c(-1)}  &   \multicolumn{3}{c}{$\xleftrightarrow{\hspace*{0.7cm}c(1)\hspace*{0.7cm}}$}   &  \multicolumn{3}{c}{$\xleftrightarrow{\hspace*{0.7cm} { c(2)}\hspace*{0.7cm}}$}   & \\
\begin{block}{cc(c|ccc|ccc|c)}
  &\multirow{1}{*}{$c(-1)\updownarrow$} & 2 & 0 & \dots & 0 & -1 & \dots & -1 & \dots \\ \cline{3-10}
    &\multirow{3}{*}{ $ \,\,c(1)\,\,\left\updownarrow\vphantom{\displaystyle\sum_{\substack{i=1\\i=0}}}\right.$}& 0 & -2 & \dots & -2 & -3 & \dots & -3 &   \\
      & & \vdots & \vdots & \ddots & \vdots & \vdots & \ddots & \vdots & \dots  \\
        & & 0 & -2 & \dots & -2 & -3 & \dots & -3 &   \\ \cline{3-10}
         & \multirow{3}{*}{ $\,\,c(2)\,\, \left\updownarrow\vphantom{\displaystyle\sum_{\substack{i=1\\i=0}}}\right.$}  & -1 & -3 & \dots & -3 & -4 & \dots & -4 &   \\
    && \vdots & \vdots & \ddots & \vdots & \vdots & \ddots & \vdots & \dots  \\
        && -1 & -3 & \dots & -3 & -4 & \dots & -4 &   \\ \cline{3-10}
         && \vdots &  & \vdots &  &  & \vdots & \vdots &   \\
\end{block}
\end{blockarray}}\;\;,
\end{equation}
so that  $\mathfrak m \cong\mathfrak g(A) / \mathfrak z$ where $\mathfrak g(A)$ is the Borcherds algebra associated to $A$  and $\mathfrak{z}$ is the center of $\mathfrak{g}(A)$. 
The numbers $c(j)$ are the Fourier coefficients of $q^j$  in the modular function $J(q)$, given by Equation~\eqref{Jfun}.

 We recall that for  Kac--Moody algebras or for Borcherds algebras $\mathfrak g$, there are two distinct types of roots.
A root $\alpha$ of $\mathfrak g$ is called {\it real} if $\|\alpha\|^2> 0$ and {\it imaginary} if  $\|\alpha\|^2\leq 0$. For Kac--Moody algebras, the simple roots are all real. However, Borcherds algebras may contain both real and imaginary simple roots. The existence of imaginary root generators profoundly influences the structure of these Lie algebras, as is the case for $\mathfrak{m}$.

 The Monster Lie algebra $\mathfrak{m}$ has root space decomposition (see \cite{BoInvent}, \cite{JLW})
\begin{equation*}\label{mdecomp}\mathfrak{m}=\left(\bigoplus_{m,n<0}\mathfrak{m}_{(m,n)} \oplus \mathfrak{m}_{(-1,1)}\right)\oplus \mathfrak{m}_{(0,0)}\oplus \left(\mathfrak{m}_{(1,-1)}\oplus\bigoplus_{m,n>0}\mathfrak{m}_{(m,n)}\right),\end{equation*}
where $(m,n)\in \textrm{II}_{1,1}$.  Recall \cite{FLM} that  $\M$ acts on $V^\natural$ by vertex operator algebra automorphisms and that $\Aut(V^\natural)$ is precisely $\M.$
The action of $\M$ on $V^\natural$ induces an action of $\M$ on $\frak m$. In particular, the $\M$-action on $\mathfrak{m}$ gives rise to an $\M$-action on each of the root spaces $\mathfrak{m}_{(1,n)}$ and $ \mathfrak{m}_{(-1,-n)}$ and $\M$ acts trivially on $\frak h=\mathfrak{m}_{(0,0)}$, which we take to be the Cartan subalgebra of $\frak m$  (\cite{BoInvent}, \cite{JLW}).

Borcherds \cite{BoInvent} used the No-ghost Theorem    to obtain an $\M$-module isomorphism between homogeneous subspaces  $\frak m_{(m,n)}$ for $(m,n)\neq (0,0)$ of the Monster Lie algebra and the the homogeneous subspace $V^\natural_{mn+1}$ of $V^\natural$, where $\frak h=\mathfrak{m}_{(0,0)}\cong \mathbb{R}\oplus \mathbb{R}$ is a trivial $\mathbb{M}$-module. %

We now discuss our results. In \cite{JLW}, the authors gave elements of the vertex algebra $V$ corresponding to an $\mathfrak{sl}_2$-triple associated to the unique positive real simple root of $\frak m$ and extended it to a set of generators for a $\mathfrak{gl}_2$ subalgebra of $\frak m$. They also that  noted that $\mathbb{M}$ acts trivially on  this $\mathfrak{gl}_2$ subalgebra.

In this work, we construct elements of the vertex algebra $V$ corresponding to  $\mathfrak{sl}_2$-triples associated to each imaginary simple root of $\frak m$. Since the Cartan subalgebra of $\frak m$ is two dimensional, these $\mathfrak{sl}_2$-triples  conveniently extend to a bases for  $\frak{gl}_2$ subalgebras, denoted $\mathfrak{gl}_2(j,u,v)$, associated to each imaginary simple root $(1,j)$ and each suitable pair of primary vectors $u,v$ in {$V^\natural$}.  This requires the existence of pairs $u,v$ satisfying some natural conditions, which we prove (Corollary~\ref{primary} and Appendix~\ref{sectiongen}). 

Let $P^\natural_{j}$ denote the subspace of primary vectors in $V^\natural_{j}$. Then $P^\natural_{j}$ is an $\M$-submodule of $V^\natural_{j}$ (Lemma~\ref{submodprim}).
It is straightforward to show that the inclusion $P^\natural_{j}\subseteq V^\natural_{j}$ is strict when $j>0$ (Corollary~\ref{proper}).

 We show that the action of $\mathbb{M}$ on $V^\natural$ as in \cite{FLM}, restricted to the subspace $P^\natural_{j+1}$ of primary vectors in $V^\natural_{j+1}$, induces an $\M$-action on the set of $\mathfrak{gl}_2(j,u,v)$ subalgebras corresponding to a fixed imaginary simple root. (Section~\ref{ActionOfM}). 

We give a natural sufficient condition for the non-triviality of this action; namely if $P^\natural_{j+1}$ has  more linearly independent primary vectors than the multiplicity of the trivial $\M$-module in the decomposition of $V^\natural_{j+1}$ as a direct sum of irreducible $\M$-modules, 
then the action of $\M$ on the family of $\mathfrak{gl}_2(j,u,v)$ subalgebras, corresponding to a fixed imaginary simple root $(1,j)$,  is non-trivial (Theorem~\ref{non-trivialactionR}). Using the isomorphism of $V_{j+1}^\natural$ with  $\mathfrak{m}_{(1,j)}$, and $\mathfrak{m}_{(-1,-j)}$ as $\M$-modules via the No-ghost Theorem, we use the above sufficient condition to prove that this $\M$-action is non-trivial for  $0<j<100$ (Section~\ref{ActionOfM}). To achieve this, we  used PARI/GP to compute $\dim P^\natural_{j+1}$ and the multiplicity of the trivial representation  (Appendices \ref{sectiongen} and \ref{NonTriv}).  We conjecture that this action is non-trivial for all $j>0$. %

{A motivation for this work is to understand the relationship between the action of $\M$ on $\frak m$ and a recent construction of a Lie group analog by \cite{CJM} for $\frak m$. In \cite{CJM}, the authors constructed  a Lie group analog $G(\frak m)$ for $\frak m$ given by generators and relations, where $G(\frak m)$ is generated by $\GL_2$ subgroups corresponding to positive roots, both real and imaginary. A drawback of this approach is that it does not reflect the action of the Monster  group $\mathbb M$ on $\mathfrak m$.  In \cite{ACJKM1}, we  construct a group $\mathfrak{G}(\frak m)$ which  has defining relations induced from  the $\mathfrak{gl}_2(j,u,v)$ subalgebras  constructed here and other subalgebras of $\frak m$. The group $\mathfrak{G}(\frak m)$ is generated by a subgroup $\GL_2(-1)$ corresponding to the unique real simple root $(1,-1)$  and an infinite family of subgroups $\GL_2(j,u,v)$ corresponding to imaginary simple roots $(1,j)$ for $j\geq 1$    and a suitable pair of primary vectors $u,v\in V^\natural$. The Monster group $\mathbb{M}$ acts on the set of $\GL_2(j,u,v)$ subgroups of $\mathfrak{G}(\frak m)$ and fixes $\GL_2(-1)$. In future work, we hope to investigate open questions about the relationship between the group $\mathfrak{G}(\frak m)$, the Monster group $\M$, the role of primary vectors in $V^\natural$ as in Section \ref{primary1} and irreducible $\M$-submodules of ~$V^\natural$.}

\section*{Acknowledgments}The authors are extremely grateful to John Duncan, Yi-Zhi Huang and Jim Lepowsky for their interest in this work and helpful comments. The authors also thank the anonymous referee for their comments. This work began at the WINART3 Workshop at the Banff International Research Station. The authors thank  BIRS and the organizers of the workshop.

{\section{Preliminaries}\label{prelim}
{We use the standard notation from the theory of conformal vertex algebras and vertex operator algebras (as in~\cite{FLM}, ~\cite{FHL} and ~\cite{HLZ}). Throughout this paper, following \cite{BoInvent}, we  take our base field to be $\mathbb{R}.$
Following \cite{HLZ}, a {\emph{conformal vertex algebra}} is {a }vertex algebra (in the sense of \cite{BoPNAS}, see also \cite{LL}) equipped with a $\mathbb{Z}$-grading and with a conformal vector satisfying the usual compatibility conditions (see \cite{HLZ}, Section 2). 

If $V$ also satisfies the two {\emph{grading restriction conditions}}: $V_{n}=0$ for $n<<0$  and $\dim V_{n}<\infty$ for $n\in \mathbb{Z},$ then $V$ is a {\emph{vertex operator algebra}} in the sense of \cite{FLM} and \cite{FHL}. 
{Note that, since our base field is $\mathbb{R},$ the conformal vertex algebras in this paper are vector spaces over $\mathbb{R}$.%
}

{Let ${{L=\textrm{II}_{1,1}}}$ be the even unimodular Lorentzian lattice of rank 2, which is $\Z\oplus\Z$ with bilinear form $\langle\cdot,\cdot\rangle$ given by the matrix $\left(\begin{smallmatrix} 
~0 & -1\\
-1 & ~0
\end{smallmatrix}\right)$. 
Recall that the Monster Lie algebra is defined as a quotient of the {{physical space}} of the vertex algebra $V=V^\natural\otimes V_{1,1}$, where $V^\natural$ is the Moonshine module vertex operator algebra and  $V_{1,1}$ is the  lattice vertex algebra associated to $L$ \cite{FLM}, \cite{BoPNAS}. 
}
The structure of $V_{1,1}$ will play a role in the proof of our main result (Theorem~\ref{main}, Section~\ref{mainsect}). We review the construction of $V_{1,1}$
following \cite{FLM} (Chapters 5-8), \cite{LL} and \cite{DL}. %

{Let $\mathfrak{h}=L\otimes_\mathbb{Z}\mathbb{R}$. Viewing $\mathfrak{h}$ as an abelian Lie algebra,  we define the affine Lie algebra (as in \cite{FLM}, Section 7.1) $$\widehat{\mathfrak{h}}=\coprod_{n\in \mathbb{Z}}\mathfrak{h}\otimes t^n\oplus \mathbb{R}{c\mkern-7.5mu/}$$ with Lie bracket given by
\begin{equation*}
[x\otimes t^m, y\otimes t^n]=\langle x, y\rangle m\delta_{m+n,0}{c\mkern-7.5mu/}
\end{equation*} for $x, y\in \mathfrak{h}$ and $m, n\in \mathbb{Z}$ and \begin{equation*}
[{c\mkern-7.5mu/},\widehat{\mathfrak{h}}]=0.
\end{equation*}
Let $\widehat{\mathfrak{h}}^{-}=\displaystyle\coprod_{n<0}\mathfrak{h}\otimes t^n,$ let $\widehat{\mathfrak{h}}^{+}=\displaystyle\coprod_{n>0}\mathfrak{h}\otimes t^n$ and let $\widehat{\mathfrak{h}}_\mathbb{Z}=\widehat{\mathfrak{h}}^{-}\oplus \widehat{\mathfrak{h}}^{+}\oplus \mathbb{R}{c\mkern-7.5mu/}$.

Let $U(\cdot)$ and $S(\cdot)$ denote the universal enveloping algebra and symmetric algebra, respectively. Consider the induced $\widehat{\mathfrak{h}}$-module
$$M(1):=U(\widehat{\mathfrak{h}})\otimes_{U(\mathfrak{h}\otimes \mathbb{R}[t]\oplus \mathbb{R}{c\mkern-7.5mu/})}\mathbb{R}\cong S(\widehat{\mathfrak{h}}^{-}),$$ where $\mathfrak{h}\otimes \mathbb{R}[t]$ acts trivially on $\mathbb{R}$ and ${c\mkern-7.5mu/}$ acts as $1$. As in \cite{FLM}, we will use the notation $\alpha(n)$ for the action of $\alpha\otimes t^n$ on $S(\widehat{\mathfrak{h}}^{-}),$ where $\alpha\in \mathfrak{h}$ and $n\in \mathbb{Z}.$

Let $\widehat{L}$ be the central extension $$1\longrightarrow\langle\kappa\rangle\longrightarrow\widehat{L} \bar{\longrightarrow} L\longrightarrow 1$$
of $L$ by the  group $\langle \kappa\rangle=\langle \kappa \mid \kappa^2=1\rangle$ of order 2, 
with commutator map
$$c_0:L\times L\rightarrow \mathbb{Z}/2\mathbb{Z}$$ given by $$c_0(\alpha, \beta)=\langle \alpha, \beta\rangle+2\mathbb{Z}$$ \cite{FLM} equation (7.2.22) so that $$aba^{-1}b^{-1}=\kappa^{c_0(\bar{a}, \bar{b})}$$ for $a, b\in \widehat{L}$, and images $\bar{a}, \bar{b}$ in $L$ of $a,b$ respectively. 
Note that by Proposition 5.2.3 in \cite{FLM}, $c_0$ determines $\widehat{L}$ uniquely, up to equivalence of central extensions of $L$ by $\langle \kappa\rangle$,  in the sense of \cite{FLM} (5.1.1).
Viewing $\mathbb{R}$ as a $\langle \kappa \rangle$-module such that $\kappa\cdot 1=-1,$ we define $\mathbb{R}\{L\}$ to be the induced $\widehat{L}$-module:
\begin{equation*}
\mathbb{R}\{L\}=\mathbb{R}[\widehat{L}]\otimes_{\mathbb{R}[\kappa]}\mathbb{R}=\mathbb{R}[\widehat{L}]/(\kappa-(-1))\mathbb{R}[\widehat{L}]\cong \mathbb{R}[{L}] \quad\text{ (linearly)}
\end{equation*}\cite{FLM}  (7.1.18).  Given $a\in \widehat{L}$ we denote the image of $a$ in $\mathbb{R}\{L\}$ by $\iota(a).$ Then
\begin{equation}\label{aacts}
a\cdot \iota(b)=\iota(ab)
\end{equation}and
\begin{equation}\label{kappaacts}
\kappa\cdot \iota(b)=\iota(\kappa b)=-\iota(b),
\end{equation} \cite{FLM} (7.1.19-7.1.20).
We let $V_{1,1}=S(\hat{\mathfrak{h}}^{-})\otimes \mathbb{R}\{L\}$ and we view $S(\hat{\mathfrak{h}}^{-})$ as a trivial $\widehat{L}$-module and $\mathbb{R}\{L\}$ as a trivial $\widehat{\mathfrak{h}}_\mathbb{Z}$-module.

Given $h\in \mathfrak{h}$ and $a\in \widehat{L},$ define
\begin{equation}\label{hacts}
h\cdot \iota(a)=\langle h, \bar{a}\rangle \iota(a)
\end{equation} and \begin{equation}
x^h\cdot \iota(a)=x^{\langle h, \bar{a}\rangle}\iota(a), 
\end{equation} \cite{FLM} (7.1.33) and (7.1.35).
Then $\widehat{L},$ $\widehat{h}_\mathbb{Z},$ $\mathfrak{h},$ and $x^h$ act on $V_{1,1}$ by acting on $S(\hat{\mathfrak{h}}^{-})$ or $\mathbb{R}\{L\}$ as defined above.

  As in \cite{FLM} Sections 8.4 and 8.5, the vertex operators for $V_{1,1}$ are given as follows. For $\alpha\in \mathfrak{h},$ define (\cite{FLM} (8.4.7))
\begin{equation}\label{vertexoph}\alpha(x)=\sum_{n\in \mathbb{Z}}\alpha(n)x^{-n-1}\end{equation} and given $a\in \widehat{L},$ define (\cite{FLM} (8.4.17))
\begin{equation}\label{vertexopL}Y(a, x)={}_{\circ}^{\circ}e^{\int (\bar{a}(x)-\bar{a}(0)x^{-1})}ax^{\bar{a}}{}_{\circ}^{\circ},\end{equation} where ${}_{\circ}^{\circ}~{}_{\circ}^{\circ}$ indicates that we use the normal ordering procedure as defined in \cite{FLM} Section 8.4.

Let $a\in \widehat{L},$ $\alpha_1, \cdots, \alpha_k\in \mathfrak{h},$ $n_1, \cdots, n_k\in \mathbb{Z}_+$ and define
$$v=\alpha_1(-n_1)\cdots \alpha_k(-n_k)\otimes \iota(a)=\alpha_1(-n_1)\cdots\alpha_k(-n_k)\cdot \iota(a)\in V_{1,1}.$$ 

Then define
\begin{equation}\label{verteop}Y(v, x)={}_{\circ}^{\circ}\left(\frac{1}{(n_1-1)!}\left(\frac{d}{dx}\right)^{n_1-1}\alpha_1(x)\right)\cdots \left(\frac{1}{(n_k-1)!}\left(\frac{d}{dx}\right)^{n_k-1}\alpha_k(x)\right)Y(a, x){}_{\circ}^{\circ}\end{equation} (\cite{FLM} (8.5.5)). This gives $V_{1,1}$ the structure of a conformal vertex algebra with vacuum vector $\iota(1)$ and conformal vector $\omega$ given by \begin{equation}\label{confV11}\omega=\frac{1}{4}\alpha(-1)^2+\frac{1}{4}\beta(-1)^2,\end{equation} where $\alpha=(1, 1)$ and $\beta=(1,-1)$, \cite{FLM} (8.7.2), (8.7.3). 
 Given $\alpha_1, \cdots, \alpha_k\in \mathfrak{h}$, $n_1,\cdots, n_k\in \mathbb{Z}_+$, and {$c\in \widehat{L}$, we have
\begin{equation}\label{grading}
L(0)(\alpha_1(-n_1)\cdots \alpha_k(-n_k)\cdot \iota(c))=\left( \frac{1}{2}\langle {\overline{c}, \overline{c}}\rangle+n_1+\cdots +n_k\right) \alpha_1(-n_1)\cdots \alpha_k(-n_k)\cdot{\iota(c)},
\end{equation}
where $L(0)$ denotes the zero mode of $\omega,$ that is, the coefficient of $x^{-2}$ in the vertex operator for $\omega.$
The following formula follows from \eqref{vertexopL} and the definitions given above for the operators appearing in \eqref{vertexopL}. It will be useful in the proof of Theorem \ref{main}.

Given $a, b\in \widehat{L}$, we have \cite{FLM} (8.5.11)
 \begin{equation}\label{V11op1}
 Y(\iota(a), x)\iota(b)=\exp\left(\sum_{n\ge 1}\frac{\bar{a}(-n)}{n}x^n\right)\iota(ab)x^{\langle \bar{a}, \bar{b}\rangle}.
 \end{equation} }

\begin{remark}\label{notVOA}%
As in \cite{BoInvent}, $V_{1,1}$ fails to satisfy the grading restriction conditions: for any $j\in \mathbb{Z},$ $V_{1,1}$ has nonzero vectors of weight $j$, 
and $V_{1,1}$ has infinite-dimensional subspaces of homogeneous weight.
It follows that $V=V^\natural\otimes V_{1,1}$ also fails the grading restriction conditions and hence is not a Vll operator algebra.
\end{remark}

\section{The Monster Lie algebra $\mathfrak m$}\label{Monster}

Let $V^{\natural}$ be  the Moonshine module vertex operator algebra. 
Recall that we use the conformal weight grading on $V^\natural$ so that $V^\natural=\bigoplus_{n=0}^\infty V_{n}^\natural$.  As in \cite{BoInvent}, we let $V=V^\natural\otimes V_{1,1}$.

Given a conformal vertex algebra  with conformal vector $\omega$, we use $L(j)$, $j\in \mathbb{Z}$, to denote the modes of the conformal vector, that is, the coefficients in the vertex operator $Y(\omega, x)=\sum_{j\in \mathbb{Z}}L(j)x^{-j-2}$.
  A vector $v$ in a conformal vertex algebra is called {\it primary} if it satisfies 
$L(j)v=0$ for any $j> 0$. In addition, if $L(0)v=nv$ for $n\in\Z$ 
then $v$ is \emph{primary of weight $n$} \cite{FLM}, (8.10.12).

{The vertex operator algebra} $V^\natural$ has a unique (up to multiplication by a constant) symmetric invariant bilinear form (in the sense of \cite{FHL} Section 5, see also \cite{Li}), which can be normalized to be positive definite \cite{FLM} (see Proposition ~\ref{formdef}).

 The following constructions of the Monster Lie algebra were given in \cite{BoInvent}, (see also \cite{JurJPAA}).

\begin{enumerate} 
\item  We have $\mathfrak m=P_1/R$  where 
$$P_1=\{\psi\in V^\natural\otimes V_{1,1}\mid L(0)\psi=\psi,\ L(j)\psi=0, \ j> 0\}$$
is the subspace of $V=V^\natural\otimes V_{1,1}$ of primary vectors of weight 1, and $R$ denotes the radical of a natural bilinear form on $V$ (see Section~\ref{voa}). {The subspace $P_1$ is called the {\emph{physical space}} of the {conformal }vertex algebra~$V$.}

\item Let $\frak g(A)$ be the Lie algebra associated to  $A,$ where $A$ is given by Equation~\ref{mdecomp}. The Monster Lie algebra is $\mathfrak m\cong\mathfrak{g}(A)/\mathfrak{z},$ where $\mathfrak{z}$ is the center of~$\frak g(A)$. 
\end{enumerate}
These two constructions give isomorphic Lie algebras (\cite{BoInvent}, \cite{BoGKMA}, \cite{JurJPAA}). The proof of the existence of an isomorphism between {the constructions described in }(1) and (2) uses, in part, the No-ghost Theorem of \cite{GT}.

The Monster Lie algebra has the triangular decomposition $\mathfrak m= \mathfrak{n}^-\oplus \mathfrak{h}\oplus \mathfrak{n}^+$ where $\mathfrak{n}^\pm$ are direct sums of all positive (respectively negative) root spaces, and $\mathfrak{h}$ is the Cartan subalgebra of~$\mathfrak m$.

For $j\in\Z$, {recall} the definition of  $c(j)$ above. Define index sets 
$$I^{\re}=\{(-1,1)\}, \quad I^{\im} = \{(j,k)\mid j\in\N,\; 1 \leq k\leq c(j)\}, \quad {{\text{ and }}I  =I^{\re} \sqcup I^{\im}.}$$

We will use the following presentation of $\frak{m}$ (for comparison, see \cite{BoGKMA}, \cite{JurJPAA}). The Monster Lie algebra $\mathfrak{m}$ may be generated by the set
$$\{e_{-1},f_{-1},h_{1},h_2\}\cup \{e_{jk},\ f_{jk}\mid (j,k)\in I^{\im}\}, $$
where the  generators indexed over $I^{\re}$ are denoted
$e_{-1}$, $f_{-1}$, the  generators indexed over $I^{\im}$ are denoted $\{e_{jk},\ f_{jk}\mid (j,k)\in I^{\im}\}$, and $h_1$ and $h_2$ form a basis for the Cartan subalgebra~$\mathfrak{h}$. The defining relations with respect to this generating set are
\begin{align*}
\tag{M1}\label{Mhh} [h_{1},h_2]&=0,\\
\tag{M2a}\label{Mhe-} [h_1,e_{-1}] &= e_{-1},&              
              [h_2, e_{-1}] &= -e_{-1},\\
\tag{M2b}\label{Mhe} [h_1,e_{jk}]&= e_{jk},&
              [h_2, e_{jk}] &= j e_{jk},\\
\tag{M3a} \label{Mhf-} [h_1,f_{-1}] &= -f_{-1},&
             [h_2,f_{-1}] &= f_{-1},\\
\tag{M3b}\label{Mhf} [h_1,f_{jk}]&= - f_{jk},&
              [h_2,f_{jk}] &= -j f_{jk},\\
\tag{M4a}\label{Me-f-} [e_{-1},f_{-1}]&=h_1-h_2, \\
\tag{M4b}\label{Mef-} [e_{-1},f_{jk}]&=0,  &[e_{jk},f_{-1}]&=0,\\
\tag{M4c}\label{Mef} [e_{jk},f_{pq}]&=-\delta_{jp}\delta_{kq} \left(jh_1 + h_2\right),\\ 
\tag{M5}\label{Mee}
  (\ad e_{-1})^j e_{jk}&=0,&\qquad (\ad f_{-1})^j f_{jk}&=0,
\end{align*}
for all $(j,k),\,(p,q) \in {I}^{\im}$. %

As in \cite{JLW}, we may identify the root lattice of $\mathfrak{m}$ with the Lorentzian lattice $\textrm{II}_{1,1},$ that is, $\Z\oplus\Z$ equipped with the bilinear form given by the matrix~$\left(\begin{smallmatrix}  ~0 & -1\\-1 & ~0
 \end{smallmatrix}\right)$.  
The simple roots are $\alpha_{-1}$ and $\alpha_{jk}$ for $(j,k)\in I^{\im}$. Let $Q$ denote the root lattice for $\mathfrak g(A)$ $$Q := \mathbb{Z}\alpha_{-1}\oplus\bigoplus_{(j,k)\in I^{\im}}\Z\a_{jk}.$$ In Figure \ref{fig:MonsterSimpleRoots}, blue nodes on the blue hyperbola denote the real roots with squared norm 2, and imaginary roots with squared norms $\leq -2$ are denoted by red nodes on the red hyperbolas. 
 
\begin{figure}[h]
\includegraphics[height=6.8cm]{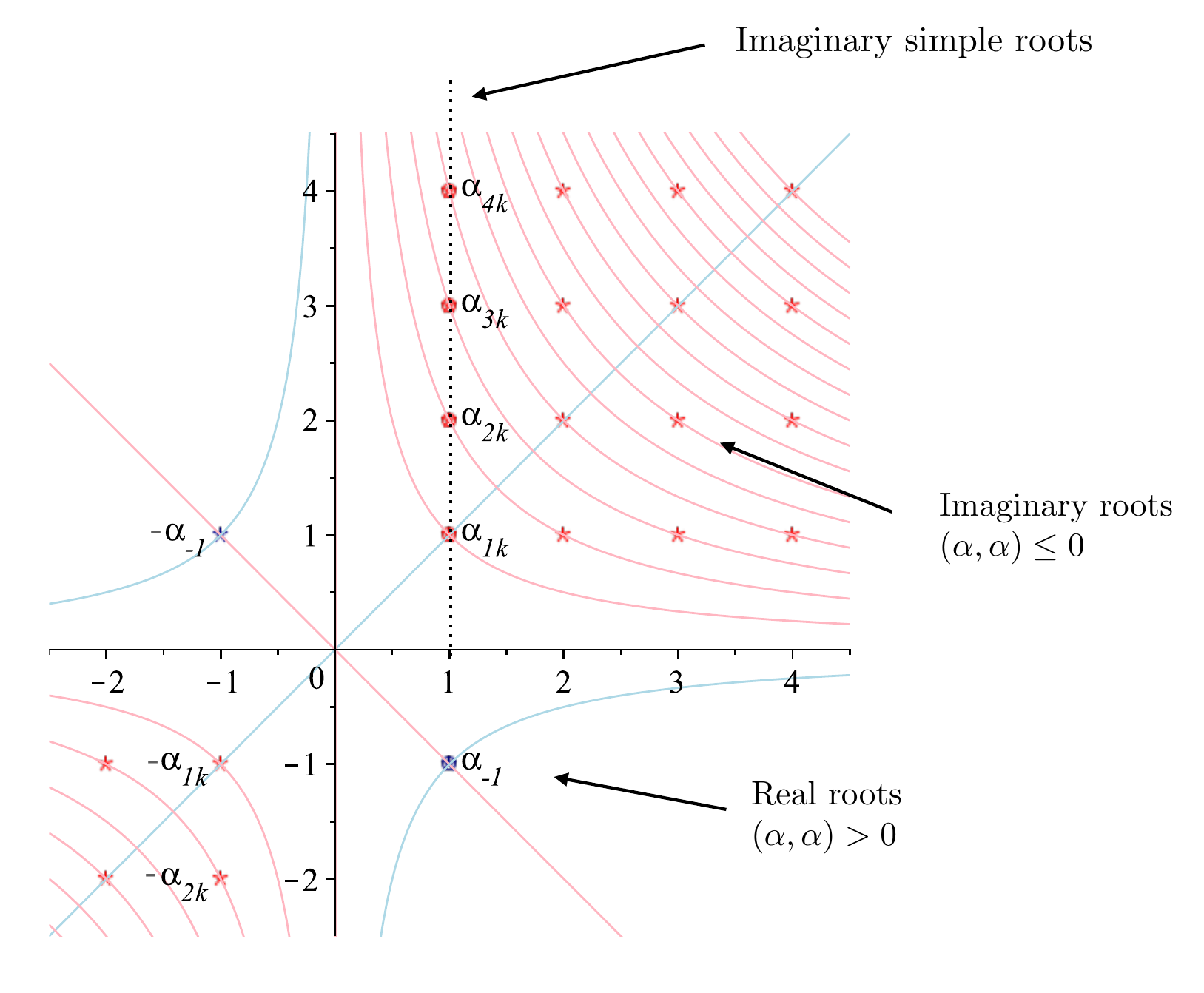} 
\caption{{Roots for the Monster Lie algebra $\frak m$.}}
\label{fig:MonsterSimpleRoots}
 \end{figure}
 We have a surjective $\Z$-linear \emph{specialization} map 
 \begin{align*} Q &\to \textrm{II}_{1,1}\\ \a&\mapsto\overline\a
 \end{align*}
 from the root lattice $Q$ of $\mathfrak{g}(A)$ to the root lattice $\textrm{II}_{1,1}$ of $\mathfrak{m}$ with
$$\overline{\a_{-1}}=(1,-1)\quad\text{and}\quad\overline{\a_{jk}}=(1,j),$$
(see Corollary 5.4 of \cite{JurJPAA}). 

For $j\in\N$, we obtain (specialized) root spaces 
\begin{align*}
 \mathfrak{m}_{(1,-1)}&=\R e_{-1}, &
 \mathfrak{m}_{(-1,1)}&= \R f_{-1},\\
 \mathfrak{m}_{(1,j)}&=\bigoplus_{k=1}^{c(j)}\,\R e_{jk}, &
 \mathfrak{m}_{(-1,-j)}&=\bigoplus_{k=1}^{c(j)} \,\R f_{jk}.
 \end{align*}
We have $\dim(\mathfrak{m}_{(m,n)})=c(mn)$ (\cite{BoInvent}, \cite{JurJPAA}). %
 The Weyl transformation $w_{\re}$ on {$\textrm{II}_{1,1}$} with respect to the unique positive real root $\alpha_{-1}$
 in the blue line in Figure~\ref{fig:MonsterSimpleRoots} is $w_{\re}:(n, m) \mapsto (m, n)$.

 We call $\{e_{-1},\ f_{-1},\ h_1-h_2\}$ the {\it $\frak{sl}_2$-triple} corresponding to the real simple root $(1,-1)$. The subalgebra of $\frak m$ with basis $e_{-1}$, $f_{-1}$, $h_1$, $h_2$
 is isomorphic to $\mathfrak{gl}_2$ and is denoted $\mathfrak{gl}_2(-1)$. {Similarly~$\{e_{jk}$, $f_{jk}, -(jh_1+h_2)\}$ }generates an $\frak{sl}_2$ subalgebra corresponding to the imaginary simple root $(1,j)$. For a description of  $\frak{gl}_2$ subalgebras of $\mathfrak m$ for all imaginary root vectors $(m,n)$ with $m>1$, see \cite{CJM}.

As discussed above,  the Monster Lie algebra is an example of a Borcherds algebra. The following theorem
(\cite{BoInvent}, see also \cite{JurJPAA}) gives criteria for a Lie algebra to be a Borcherds algebra. 

\begin{theorem}\label{GKMA}(\cite{BoInvent}, \cite{JurJPAA}) Let $\mathfrak g$ be a Lie algebra satisfying the
following conditions:
\begin{enumerate}
 \item  $\mathfrak g$ can be $\mathbb Z$-graded as $\coprod_{i\in {\mathbb
Z}} \mathfrak g_i$, where $\mathfrak g_i$ is finite dimensional if $i \neq 0$, and
$\mathfrak g$ is diagonalizable with respect to $\mathfrak g_0$. 
\item  $\mathfrak g$ has an involution $\omega$ which maps $\mathfrak g_i$
onto $\mathfrak g_{-i}$ and {acts as $-1$} on 
$\mathfrak g_0$. In particular, $\mathfrak g_0$ is abelian.
\item  $\mathfrak g$ has a Lie algebra-invariant bilinear form $(\cdot,\cdot)$,
invariant under $\omega$, such that $\mathfrak g_i$ and $\mathfrak g_j$ are
orthogonal if $i \neq -j$, and such that the form $(\cdot,\cdot)_0$,
defined by $(x,y)_0 = -(x, \omega (y))$ for $x,y \in \mathfrak g$, is positive
definite on $\mathfrak g_m $ if $m \neq 0$.
\item  $\mathfrak g_0 \subset [\mathfrak g,\mathfrak g]$. 
\end{enumerate}

Then $\mathfrak{g}$ is a Borcherds algebra, that is, there is a Borcherds Cartan matrix $A$ such that $\mathfrak{g}$ is isomorphic to $\mathfrak{g}(A)/\mathfrak{z}$ for some central subalgebra $\mathfrak{z}$. 
\end{theorem}
For the Monster Lie algebra, we have $\mathfrak{m}=\mathfrak{g}(A)/\mathfrak{z}$ where $\mathfrak{z}$ is the  center of 
$\mathfrak{g}(A)$ and $A$ is given in Equation~\eqref{mdecomp}.

Using the No-ghost Theorem, Borcherds proved \cite{BoInvent} (see also \cite{FLM} and \cite{JLW}), that 
$\mathfrak m_{(m,n)}\cong V^\natural_{mn+1}$ as  $\mathbb{M}$-modules for $(m,n)\neq (0,0)$, and $\mathfrak m_{(0,0)}\cong\mathfrak h$ which is a trivial $\mathbb{M}$-module. This gives a representation $\mathbb{M}   \to    \Aut(\mathfrak m)$ acting trivially on $\mathfrak h$ and preserving root spaces $(m,n)$.

\section{The vertex algebraic construction of $\mathfrak m$ }\label{voa}

{In this section, we review the construction of the Monster Lie algebra as a quotient of the physical space $P_1$ of $V=V^\natural\otimes V_{1,1}$. We begin with a review of invariant bilinear forms on {conformal vertex algebras}.}
The following definition from \cite{FHL} holds for conformal vertex algebras over an arbitrary field of characteristic $0$. A bilinear form $(\cdot, \cdot)_V$ on a conformal vertex  algebra $V$ is  invariant if
\begin{equation}\label{inv}
(Y(v, x)w_1, w_2)=(w_1, Y(e^{xL(1)}(-x^{-2})^{L(0)}v, x^{-1})w_2)
\end{equation} for $v, w_1, w_2\in V.$ {Here $(-x^{-2})^{L(0)}$ is defined by $(-x^{-2})^{L(0)}v=(-x^{-2})^{\text{wt }v}v$ for $v\in V$ of homogeneous weight.}
Recall (\cite{FHL} Chapter 5, \cite{Li} Theorem 3.1) that given a {inv} operator algebra $V$ such that $\dim V_0=1$ and $L(1)V_1=0,$ there exists a unique symmetric invariant bilinear form up to multiplication by a constant.

For completeness, we include the following proposition, which will be used in the proof of our main theorem (Theorem \ref{main}).

\begin{proposition}(\cite{FHL}, \cite{Li})\label{formdef}   Let $V$ be a {inv} operator algebra  with $\dim {V_0}=1$ and $L(1)V_1=0$. Then the unique (up to multiplication by a constant) symmetric invariant bilinear form $(\cdot, \cdot)_V$ on $V$ is defined by $$(u,v)_V=0$$ if~$\text{\rm wt }u\ne \text{\rm wt }v$ and $$(u,v)_V{\bf 1}=-\Res_xx^{-1}Y(e^{xL(1)}(-x^{-2})^{L(0)}u, x^{-1})v,$$ if~$\text{\rm wt }u= \text{\rm wt }v$.

\end{proposition}

{\begin{proof}

 Let $(\cdot, \cdot)$ denote the symmetric invariant bilinear form on $V$ normalized such that $({\bf 1}, {\bf 1})=-1$. We will show that $(\cdot, \cdot)=(\cdot, \cdot)_V$, where $(\cdot, \cdot)_V$ is as defined in the statement of Proposition \ref{formdef}.

Given $u, v\in V$, applying Equation \eqref{inv}, along with $L(1)\omega=0$ and $L(0)\omega=2\omega,$ we have 
\begin{eqnarray*}
(L(0)u,v)&=&\Res_x x\left(\sum_{n\in \mathbb{Z}}L(n)x^{-n-2}u, v\right)
=\Res_x x\left(Y(\omega, x)u, v\right)\\
&=&\Res_x x\left(u, Y(e^{xL(1)}(-x^{-2})^{L(0)}\omega, x^{-1}v\right)
=\Res_x x^{-3}\left(u, Y(\omega, x^{-1})v\right)_V\\
&=&\Res_x x^{-3}\left(u, \sum_{n\in \mathbb{Z}}L(n)x^{n+2}v)=(u, L(0)v\right).
\end{eqnarray*} Therefore, if $\text{wt }u\ne \text{wt }v$, $(u, v)=0=(u, v)_V$.

Next, suppose that $\text{wt }u= \text{wt }v$. Using the creation property followed by Equation \eqref{inv}, we have
    \begin{eqnarray*}(u, v)&=& \text{Res}_x x^{-1}(Y(u,x){\bf{1}}, v)=\Res_x x^{-1}({\bf 1}, Y(e^{xL(1)}(-x^{-2})^{L(0)}u, x^{-1})v).\end{eqnarray*}
    Since $\text{wt }u= \text{wt }v$, $\Res_x x^{-1}Y(e^{xL(1)}(-x^{-2})^{L(0)}u, x^{-1})v$ is a multiple of ${\bf 1}$. Therefore, $(u,v)$ is determined by $({\bf 1},{\bf 1})=-1$ and we have \begin{eqnarray*}(u, v){\bf 1}&=& ({\bf{1}}, {\bf{1}})\Res_x x^{-1}Y(e^{xL(1)}(-x^{-2})^{L(0)}u, x^{-1})v\\&=&-\Res_xx^{-1}Y(e^{xL(1)}(-x^{-2})^{L(0)}u, x^{-1})v=(u, v)_V{\bf 1}.\end{eqnarray*}
\end{proof}}
{In \cite{JurJPAA}, Lemma 6.1, it was proven that there exists   
 a nonzero symmetric invariant bilinear form $(\cdot,\cdot)_{V_{1,1}}$ on the conformal vertex algebra $V_{1,1}$  over $\R$, which is unique up to multiplication by a constant. %
 }

Frenkel, Lepowsky, and Meurman defined a positive definite bilinear form $(\cdot,\cdot)$ on $V^\natural$ \cite{FLM}. It was proven in \cite{JurJPAA}, Lemma 6.2 that 
$(\cdot,\cdot)$  is invariant in the sense of Equation \eqref{inv}. Since invariant bilinear forms on $V^\natural$ are unique up to multiplication by a constant, the positive definiteness of $(\cdot,\cdot)$ implies that every nonzero invariant bilinear form on $V^\natural$ is either positive definite or negative definite. 

Applying Proposition~\ref{formdef} to $V^\natural$, the bilinear form $(\cdot,\cdot)_{V^\natural}$ on $V^\natural$  satisfies $({\bf 1},{\bf 1})=-1$. Hence it is negative definite. This choice will be convenient for the proof of Theorem~\ref{main} ({\it c.f.} \cite{JurJPAA}).

 The product of the forms $(\cdot,\cdot)_{V^\natural}$ and $(\cdot,\cdot)_{V_{1,1}}$ defines a symmetric invariant bilinear form, $(\cdot, \cdot)_V$, on the tensor product $V=V^\natural\otimes V_{1,1}$ (see \cite{JurJPAA}, page 258).

 Recall from Section \ref{Monster} that 
$$P_1=\{\psi\in V^\natural\otimes V_{1,1}\mid L(0)\psi=\psi,\ L(j)\psi=0, \ j> 0\}$$
is the subspace of $V=V^\natural\otimes V_{1,1}$ of primary vectors of weight 1. %

 Let
\[P_0=\{\psi\in V^\natural\otimes V_{1,1}\mid L(j)\psi=0, \ j\ge 0\}.\]

Then $P_1/L(-1)P_0$ is a Lie algebra with bracket defined by \[[u+L(-1)P_0,v+L(-1)P_0]=u_0v+L(-1)P_0.\] Here, $u_0$ denotes the zero-mode of $u$, that is, the coefficient of $x^{-1}$ in $Y(u, x).$ Since $L(-1)P_0$ is a subspace of the radical of $(\cdot, \cdot)_V$, this induces an invariant bilinear form on the Lie algebra $P_1/L(-1)P_0$, denoted by  $(\cdot, \cdot)_{\text{Lie}}$. 

We define $\mathfrak m $ to be the quotient of $P_1/L(-1)P_0$ by the radical of the form $(\cdot, \cdot)_{\text{Lie}}.$ As in \cite{JurJPAA}, we can identify $\mathfrak{m}$ with  $P_1/R$, where $R$ is the preimage of the radical of $(\cdot, \cdot)_{\text{Lie}}$ under the quotient map $P_1 \to P_1/L(-1)P_0$.

The Lie bracket on $\mathfrak{m}$ is given by $$[u + R, v + R]:=u_0v+R,$$ for $u, v\in P_1$. The vertex operators $Y(u,z)$ for $u\in P_1$ are tensor products of the vertex operators for $V^\natural$ and $V_{1,1}$: {Given $a\in V^\natural$ and $b\in V_{1,1}$, \begin{equation}\label{tensor}
    Y(a\otimes b, x)=Y(a, x)\otimes Y(b, x)
\end{equation}} as in {\cite{FHL}, Section 2.5 and \cite{DL}, Chapter 10.} %

The {vertex} operators for $V^\natural$ are given in \cite{FLM}, Chapter 12 (using also Chapters  8, 9, 10). {See Section \ref{prelim} for a review of the vertex operators for $V_{1,1}$.}

\section{Primary vectors}\label{primary1}
Recall that a vector $v$  in a conformal vertex  algebra  $V$  is called {\it primary}  if it satisfies 
$L(j)v=0$ for any $j> 0$. In addition, if $L(0)v=nv$
then $v$ is \emph{primary of weight $n$}.

The following lemma will be useful in the proof of our main theorem (Theorem \ref{main}).
\begin{lemma}\label{existprim1}
    For primary vectors $u, v \in V_{j+1}^\natural,$ we have $(u,v)_{V^\natural}{\bf{1}}=(-1)^ju_{2j+1}v,$ where $(\cdot, \cdot)_{V^\natural}$ is the bilinear form given by Proposition \ref{formdef}.
\end{lemma}
\begin{proof}
This follows from the definition of a primary vector and the formula for the bilinear form given by Proposition \ref{formdef}.
\end{proof}

Given  the conformal vector $\omega^\natural$ of $V^\natural$ and { $\omega^{1,1}:=\omega$ where $\omega$ is defined by Equation \eqref{confV11}}, $\omega^\natural\otimes{\bf 1}+{\bf 1}\otimes \omega^{1,1}$ is a conformal vector in $V^\natural\otimes V_{1,1}$. We refer the reader to {\cite{DL}} for details about tensor products of vertex algebras. {We note that although the base field in \cite{DL} is $\mathbb{C},$ the formulas and results we will use hold over $\mathbb{R}$.}

\begin{lemma}
    If $x\in V^\natural$ is primary of weight $n$ and $y\in V_{1,1}$ is primary of weight $m$ then $x\otimes y$ is a primary vector in $V^\natural\otimes V_{1,1}$ of weight $n+m.$
\end{lemma}

\begin{lemma}(\cite{FLM}, Remark 8.7.8)
    If $\iota(c)\in V_{1,1}$ such that $c\in \widehat{L}$ with $\bar{c}=(m,n)$, then $\iota(c)$ is primary of weight $-mn.$
\end{lemma}

We denote the subspace of primary vectors of weight $j$ in $V=V^\natural\otimes V_{1,1}$ by $P_{j}$ and the subspace of primary vectors of weight $j$ in $V^\natural$ by $P^\natural_{j}.$

\begin{lemma}\label{submodprim}Let $j\ge 0$. Then $P^\natural_{j}$ is an  $\mathbb{M}$-submodule of $V^\natural_{j}$.
\end{lemma}
\begin{proof}%
Recall that  $\M$ acts on $V^\natural$ by vertex operator algebra automorphisms. Such automorphisms fix the conformal vector $\omega$ (\cite{FLM} (8.10.21)). Therefore, given any $v\in V^\natural$ and any $g\in \M$ we have, for all $n\in \mathbb{Z},$
\begin{equation*}g\cdot(L(n)v)
=L(n)g\cdot v.\end{equation*} That is, the action of $\M$ on $V^\natural$ commutes with the action of the Virasoro algebra.
In particular, if $v\in P_{j}^\natural$,
\begin{equation*}
    g\cdot (L(0)v)=g\cdot (jv)=j(g\cdot v)=L(0)g\cdot v
\end{equation*}
\begin{equation*}
     g\cdot L(j)v=0=L(j)g\cdot v,
\end{equation*}
for all $j>0$ and all $g\in \M.$
\end{proof}

Proposition \ref{Liinj} below is due to Li \cite{Li}. Generalizations of Proposition \ref{Liinj} and Proposition \ref{surj} for twisted modules were given in \cite{DLM}. In both \cite{Li} and \cite{DLM}, the base field is $\mathbb{C}$. {The proofs in \cite{Li} and \cite{DLM} which hold over $\mathbb{C}$, are included for completeness, to verify that they hold over $\mathbb{R}$.}

We will use the notion of a vertex operator algebra module as in \cite{FLM} and \cite{FHL}. In particular, given a vertex operator algebra $V=\oplus_{i\in \mathbb{Z}}V_i,$ a $V$-module $M$ is a $\mathbb{Q}$-graded vector space \[M=\oplus_{i\in \mathbb{Q}}M_i\] such that for $n\in \mathbb{Q}$ we have
$\dim M_n<\infty$, $M_n=0$ for $n$ sufficiently small, and there is a linear map
\begin{align*}
V&\rightarrow (\End M)[[x, x^{-1}]],&
v&\mapsto Y_{M}(v, x)=\sum_{n\in \mathbb{Z}}v_nx^{-n-1},
\end{align*}
where $v_n\in \End M,$ and $Y_{M}(v, x)$ denotes the vertex operator associated to $v$, such that all of the defining properties of a vertex operator algebra that make sense hold (\cite{FHL}, Definition 4.1). {Note that we use $v_n$ to denote both the coefficient of $x^{-n-1}$ in the vertex operator $Y(v, x)$ associated with $V$ and the coefficient of $x^{-n-1}$ in the vertex operator $Y_M(v, x)$ associated with the $V$-module $M$.}

The following properties in the definition of a $V$-module $M$ will be used in the proof of Proposition \ref{Liinj}:

We have 
\begin{equation*}
[L(m), L(n)]=(m-n)L(m+n)+\frac{1}{12}(m^3-m)\delta_{m+n, 0}(\text{rank }V)
\end{equation*}
where
\begin{equation*}
Y_M(\omega, x)=\sum_{n\in \mathbb{Z}}\omega_n x^{-n-1}=\sum_{n\in \mathbb{Z}}L(n)x^{-n-2}.
\end{equation*}
Given $n\in \mathbb{Q}$ and $w\in M_n,$
\begin{equation*}
L(0)w=nw.
\end{equation*}
Given any $v\in V$ and $u\in M,$
\begin{equation}\label{sufflarge}
v_nu=0
\end{equation} 
for $n$ sufficiently large.
For all $v\in V,$
\begin{equation}\label{L(-1)mod}\frac{d}{dx}Y_{M}(v, x)=Y_{M}(L(-1)v, x),\end{equation} and 
\begin{equation}\label{L(-1)brack}[L(-1), Y_{M}(v, x)]=Y_{M}(L(-1)v, x),\end{equation} {where on the right hand side of Equations \eqref{L(-1)mod} and \eqref{L(-1)brack}, $L(-1)$ acts on $V$.}

We will need the following definition from \cite{Li}, (see also \cite{DLM}). 
Let $V=\oplus_{i\in \mathbb{Z}}V_i$ be a vertex operator algebra and $M=\oplus_{i\in \mathbb{Q}}M_i$ a $V$-module. A vector $u\in M$ is called {\emph{vacuum-like}} if \begin{equation}\label{vac}a_nu=0\end{equation}for all $a\in V$ and for all $n\ge 0.$

\begin{remark}\label{vacuumrem} As is remarked in \cite{DLM}, given a vacuum-like vector $u$, taking $a=\omega$ and $n=1$ in Equation \ref{vac}, we have $\omega_1 u=L(0)u=0.$ Therefore, all vacuum-like vectors are contained in $M_0.$\end{remark}

\begin{proposition}\label{Liinj}\cite{Li}, \cite{DLM} Let $V=\oplus_{i\in \mathbb{Z}}V_i$ be a vertex operator algebra and let $M=\bigoplus_{i\in \mathbb{Q}} M_i$ be a $V$-module. Then the map from $M_{j}$ to $M_{j+1}$ given by $v \mapsto L(-1)v$ is injective unless $j=0$. In particular, the map from $V_j$ to $V_{j+1}$ given by $v\mapsto L(-1)v$ is injective unless $j=0$.
\end{proposition}
\begin{proof}Let $u\in M$ such that $L(-1)u=0.$ Let $v\in V$ such that $Y(v, x)u\ne 0.$ By Equation~\eqref{sufflarge} there exists $k\in \mathbb{Z}$ such that $v_ku\ne 0$ and $v_nu=0$ for all $n>k.$ Let such a $k\in \mathbb{Z}$ be given. Combining Equations~\eqref{L(-1)mod} and~\eqref{L(-1)brack},
we have
\[[L(-1), Y_M(v, x)]=\frac{d}{dx}Y_M(v, x),\] so that in particular,
\[[L(-1), v_{k+1}]=-(k+1)v_k.\] Thus, 
\[0=L(-1)v_{k+1}u-v_{k+1}L(-1)u=-(k+1)v_ku.\] Since $v_ku\ne 0,$ this implies that $k=-1.$ Therefore, given any $v\in V,$ $v_ku=0$ for all $k\ge 0.$ That is, $u$ is a vacuum-like vector. The first assertion of the proposition then follows from Remark \ref{vacuumrem}. The second assertion follows from the first assertion and the fact that every vertex operator algebra is a module for itself.  
\end{proof}

In order to give the proof of Proposition \ref{surj}, we first  review the definition of a contragradient module as in \cite{FHL}. Let $V=\oplus_{i\in \mathbb{Z}}V_i$ be a vertex operator algebra and let $M=\oplus_{i\in \mathbb{Q}}M_i$ be a $V$-module.  For $n\in \mathbb{Q}$, let $M_n^*$ denote the dual space of $M_n$ and define \[M'=\oplus_{n\in \mathbb{Q}}M_n^*.\] That is, $M'$ is the space of linear functionals $f:M\rightarrow \mathbb{R}$ vanishing on all but finitely many $M_n.$
Let $\langle\cdot ,\cdot\rangle$ denote the natural pairing between $M$ and $M'$ so that, given $f\in M^*$ and $v\in M,$ \[\langle f, v\rangle =f(v).\] Define adjoint vertex operators $Y'(v, x)$ for $v\in V$ by the linear map \begin{align*}
V&\rightarrow (\End M')[[x, x^{-1}]],\\v&\mapsto Y'(v, x)=\sum_{n\in \mathbb{Z}}v_n' x^{-n-1}
\end{align*} 
determined by the condition
\begin{equation}\label{adjcon}
\langle Y'(v,x)w', w\rangle=\langle w', Y(e^{xL(1)}(-x^{-2})^{L(0)}v, x^{-1})w\rangle
\end{equation} for all $v\in V$, $w'\in M',$ and $w\in M.$ We give $M'$ a $\mathbb{Q}$-grading by defining $M_n'=M_n^*$ for $n\in \mathbb{Q}.$ By \cite{FHL} Theorem 5.2.1, {the linear map defined by $v\mapsto Y'(v,x)$ and the $\mathbb{Q}$-grading give $M'$ the structure of a $V$-module. The $V$-module $M'$ is }called the {\emph{$V$-module contragradient to $M$.}}

We recall Proposition 5.3.1 in \cite{FHL}.
\begin{proposition}\label{doubledualiso}\cite{FHL}
Let $V=\oplus_{i\in \mathbb{Z}} V_i$ be a vertex operator algebra and let $M=\oplus_{i\in \mathbb{Q}}M_i$ be a $V$-module. Then there are natural identifications between the double-contragredient $M''$ and $M$, and between the double-adjoint operator $Y''(v, x)$ and $Y(v, x).$
\end{proposition}

\begin{proposition}\label{surj}\cite{DLM} Let $V=\oplus_{i\in \mathbb{Z}} V_i$ be a vertex operator algebra and let $M=\oplus_{i\in \mathbb{Q}}M_i$ be a $V$-module. Then the map from $M_j$ to $M_{j-1}$ given by $v\mapsto L(1)v$ is surjective unless $j=1$.
In particular, the map from $V_j$ to $V_{j-1}$ given by $v\mapsto L(1)v$ is surjective unless $j=1$.\end{proposition}

\begin{proof}
Given $j\in \mathbb{Z},$ and taking $v=\omega$ in Equation \eqref{adjcon}, we have
\begin{eqnarray*}{
\langle L(j)w', w\rangle }&=& \Res_x x^{j+1}\langle Y'(\omega, x)w', w\rangle = \Res_x x^{j+1}\langle w', Y(w^{xL(1)}(-x^{-2})^{L(0)}\omega, x^{-1})w\rangle \\&=&\langle w', L(-j)w\rangle.
\end{eqnarray*} Therefore, given any $V$-module $W=\oplus_{i\in \mathbb{Q}}W_i,$ the map from $W_i$ to $W_{i+j}$ given by $v\mapsto L(-j)v$
is dual to the map from
$W_{i+j}'$ to $ W_{i}'$ given by $v'\mapsto L(j)v'.$

Now take $W$ to be the contragradient module $M'$ of $M$. By Proposition \ref{Liinj}, the map from $M'_{j}$ to $M'_{j+1}$ given by $v'\mapsto L(-1)v'$ is injective unless $j=0.$ The dual of an injective map is surjective, so the map from $M_{j+1}''$ to $M_{j}''$ given by $v''\mapsto L(1)v''$ is surjective. Since the homogeneous subspaces of a $V$-module are finite dimensional, $M_{i}''=M_{i}$ for all $i\in\mathbb{Q}.$ By Theorem \ref{doubledualiso}, there is a natural identification between $Y''(\omega, x)$ and $Y(\omega, x)$. Thus, the map from $M_j$ to $M_{j-1}$ given by $v\mapsto L(1)$ is surjective unless $j=1$.
\end{proof}

{Recall that since $P_j^\natural$ is the subspace of $V^\natural_j$ consisting of all weight $j$ vectors that are annihilated by $L(j)$ for all $j\ge 1$, it is a subspace of $\ker(L(1)|_{V_j^\natural}),$ which is the subspace of $V_j^\natural$ consisting of all weight $j$ vectors annihilated by $L(1)$.}

\begin{corollary}\label{proper}
 For $j> 1$, $P^\natural_{j}$ is a proper $\M$-submodule of $V^\natural_{j}$. \end{corollary} 
\begin{proof}Let $j>1$. By Lemma \ref{submodprim}, $P_j^\natural$ is an $\M$-submodule of $V_j^\natural$. By Proposition \ref{surj}, the map from $V_j^\natural$ to $V_{j-1}^\natural$ given by $v\mapsto L(1)v$ is surjective. If $j\ne 2$, $V_{j-1}^\natural\ne \{0\},$ so $\ker(L(1)|_{V_j^\natural})\ne V_{j}^\natural$. Therefore, in this case, $\ker(L(1)|_{V_j^\natural})$ is a proper subspace of $V_j^\natural$, and since $P_j^\natural$ is a subspace of $\ker(L(1)|_{V_j^\natural})$, the module $P_j^\natural$ is a proper $\M$-submodule of $V_j^\natural.$ Finally, since the conformal vector  $\omega\in V_2^\natural$ is not primary, $P_2^\natural$ is a proper $\M$-submodule of $V_2^\natural.$
\end{proof}

\section{Vertex algebra elements for imaginary $\mathfrak{gl}_2$ subalgebras of $\mathfrak m$}\label{mainsect}

In this section we prove our main result, Theorem~\ref{main},  which is a  generalization of the following result from \cite{JLW}.

\begin{lemma}\label{JLW_lemma}\cite{JLW}  Let $L=\textrm{II}_{1,1}$, and let $a,b\in \widehat{L}$ such that  $\bar{a}=(1,1)$ and $\bar{b}=(1,-1)$. Then the elements \begin{align*}
e&={\bf 1}\otimes \iota(b)+R,&
f&={\bf 1}\otimes \iota(b^{-1})+R,\\
h&={\bf 1}\otimes\overline{b}(-1)\iota(1)+R,&
z&={\bf 1}\otimes\overline{a}(-1)\iota(1)+R
\end{align*} 
are a basis for a subalgebra $\mathfrak{gl}_2(-1)$ isomorphic to $\mathfrak{gl}_2$ in $\mathfrak m$. The subalgebra $\mathfrak{gl}_2(-1)$ is a trivial $\M$-module. The vectors $z$ and $h$ commute with each other and the elements $e,f,h$ are a basis for a subalgebra isomorphic to $\mathfrak{sl}_2$  with relations
$$[e,f]=h,\quad [h,e]=2e,\quad [h,f]=-2f.$$ 

\end{lemma}

In Theorem~\ref{main}, we  give elements in $P_1$ which project under the quotient map $P_1\to P_1/R$ to generators of $\mathfrak{gl}_2$ subalgebras for imaginary simple roots of $\frak m$.

From now on, for each $j\in \{-1,1, 2, \cdots\}$, we fix a lifting $c_j\in \widehat{L}$ of $\overline{c_j}=(1,j)\in L$, so that $\overline{{c_j}^{-1}}=(-1,-j)\in L$.

\begin{theorem}\label{main} Fix $j\in \{-1,1, 2, \cdots\}$ and {$c_j\in \widehat{L}$} such that $\overline{c_j}=(1,j)$. 
Let $u, v$ be fixed primary vectors in $V_{j+1}^\natural$ such that $u_{2j+1}v={\bf 1}$. Then $u\otimes \iota({c_j})$, $v\otimes\iota({c_j^{-1}})\in P_1$. Let $h$ and $z$ be as in Lemma \ref{JLW_lemma}. Let $t_1=(0,-1)$ and $t_2=(-1,0).$
Then
\begin{align*}
e_{j,u}&=u\otimes \iota({c_j})+R{,}&
f_{j,v}&=v\otimes\iota({c_j^{-1}})+R,\\
h_1&=\frac{h-z}{2}={\bf 1}\otimes t_1(-1)\iota(1)+R&
h_2&=-\frac{h+z}{2}={\bf 1}\otimes t_2(-1)\iota(1)+R
\end{align*}
generate a subalgebra 
$\mathfrak{gl}_2(j,u,v)$ of $\mathfrak m$ isomorphic to $\mathfrak{gl}_2$. %
In particular, {if we identify $e_{j,u}$ with $e_{jk}$ and $f_{j,v}$ with $f_{jk}$ for any fixed $k,$ $1\le k\le c(j),$ the defining relations of $\frak m$ involving this choice of generators along with $h_1$ and $h_2$, namely the relations \eqref{Mhh}, \eqref{Mhe}, \eqref{Mhf}, and \eqref{Mef}}  are satisfied. That is, we have:
\begin{align*}[{e}_{j,u}, {f}_{j,v}]&=-(j{h}_1+{h}_2)&
[{h}_{1}, {e}_{j,u}]={e}_{j,u}\\\\
[{h}_{2}, {e}_{j,u}]&=j{e}_{j,u}&
[{h}_{1}, {f}_{j,v}]=-{f}_{j,v}\\\\
[{h}_{2}, {f}_{j,v}]&=-j{f}_{j,v}&
[{h}_{1}, {h}_{2}]=0.
\end{align*}

\end{theorem}

\begin{proof}
We first compute $[e_{j,u},f_{j,v}]$ by calculating the zero mode of the vector in $P_1$ that represents the generator ${e_{j,u}}$, acting on the vector in $P_1$ that represents the generator ${f_{j,v}}$. {Define $p_k(x_1, x_2, \cdots)\in~ \R[x_1, x_2, \cdots ]$ for $k\ge 0$ by\[\exp\left(\sum_{n=1}^\infty\frac{x_n}{n}y^n\right)=\sum_{k=0}^\infty p_k(x_1, x_2, \cdots) y^k,\]  (\cite{FLM} (8.3.22) and (8.2.23)). } Equation \eqref{V11op1} is then equivalent to
{\begin{equation}\label{V11op2}
 Y(\iota(a), x)\iota(b)=\sum_{r=0}^\infty p_r(\bar{a}(-1),\bar{a}(-2), \cdots )x^r\iota(ab)x^{\langle \bar{a}, \bar{b}\rangle}.
 \end{equation}}

Applying Equation \eqref{V11op2}, we have
\begin{align*}
\begin{split}
Y(u\otimes {\iota(c_j)},x)(v\otimes{\iota(c_j^{-1})})&=\left(Y(u,x)\otimes Y({\iota(c_j)}, x)\right)(v\otimes{\iota(c_j^{-1})})\\&=Y(u,x)v\otimes Y({\iota(c_j)},x){\iota(c_j^{-1})}\\&=\sum_{i\in\mathbb{Z}}u_ix^{-i-1}v\otimes x^{\langle (1,j),(-1,-j)\rangle}\sum_{r=0}^\infty p_r(\overline{c_j}(-1),\overline{c_j}(-2)\cdots){\iota(1)}x^r\\&=\sum_{i\in\mathbb{Z}}u_ix^{-i-1}v\otimes \sum_{r=0}^\infty p_r(\overline{c_j}(-1),\overline{c_j}(-2)\cdots){\iota(1)}x^{r+2j}.
\end{split}
\end{align*}

The resulting representative of $[e_{j,u},f_{j,v}]$ is the coefficient of $x^{-1}$ in this expression. Therefore, we have
\begin{equation}\label{e0f}
[e_{j,u},f_{j,v}]=\sum_{r=0}^\infty u_{r+2j}v\otimes p_r(\overline{c_j}(-1),\overline{c_j}(-2)\cdots){\iota(1)}+R.
\end{equation}Since $u\in V_{j+1}^\natural$, $u_{r+2j}$ has weight equal to $j+1-(r+2j)-1=-r-j$ as an operator (\cite{FLM}, Remark 8.10.1). Since $v$ is also an element of $V_{j+1}^\natural$, it follows that $u_{r+2j}v\in V_{1-r}^\natural.$ Since $V_1^\natural$ is the zero subspace, and $V^\natural$ has no negative weight spaces, this implies that the only term in \eqref{e0f} that is nonzero is the term corresponding to $r=1$. Thus 
\begin{equation*}
[e_{j,u},f_{j,v}]={u}_{2j+1}{v}\otimes p_1(\overline{c_j}(-1),\overline{c_j}(-2)\cdots){\iota(1)}+R={\bf 1}\otimes\overline{c_j}(-1){\iota(1)}+R={-(jh_1+h_2)}.
\end{equation*}

We similarly compute $[h_1, e_{j,u}]$: The vertex operator for the representative of $h_1$ in $P_1$ acting on the representative of $e_{j,u}$ in $P_1$ is
\begin{align*}\begin{split}
Y({\bf 1}\otimes{t_1}(-1){\iota(1)},x)(u\otimes {\iota(c_j)})=\left({1}\otimes \sum_{i\in\mathbb{Z}}t_1(i)x^{-i-1}\right)(u\otimes {\iota(c_j)})=u\otimes \sum_{i\in\mathbb{Z}}t_1(i)x^{-i-1}{\iota(c_j)},
\end{split}\end{align*} 
having used Equations \eqref{vertexoph} and \eqref{verteop}. The coefficient of $x^{-1}$ in this expression is a representative in $P_1$ of $[h_1,e_{j,u}]$. Therefore, {applying \eqref{hacts}, }we have
\begin{align*}\label{h0e}\begin{split}
[h_1,e_{j,u}]=u\otimes t_1(0){\iota(c_j)}+R=u\otimes \langle(0,-1),(1,j)\rangle {\iota(c_j)}+R =u\otimes {\iota(c_j)}+R=e_{j,u}\end{split}\end{align*} 
and
\begin{align*}\begin{split}
[h_1,f_{j,v}]=v\otimes t_1(0){\iota(c_j^{-1})}+R=v\otimes \langle (0,-1), (-1,-j)\rangle{\iota(c_j^{-1})}+R=-v\otimes{\iota(c_j^{-1})}+R=-f_{j,v}.
\end{split}\end{align*}

Replacing $t_1$ by $t_2$ in the above calculations similarly proves 
\begin{equation*}
[{h}_{2}, {e}_{j,u}]=je_{j,u},\qquad
[{h}_{2}, {f}_{j,v}]=-j{f}_{j,v}.
\end{equation*}
Finally, the fact that 
    $[h_1,h_2]=0$ follows from the fact that $h_1$ and $h_2$ are linear combinations of $h$ and $z$ and $h$ and $z$ commute with each other.
 \end{proof}

Recall from Section~\ref{prelim} the induced $\widehat{L}$-module $\mathbb{R}\{L\},$ with the action by $\widehat{L}$ given by Equations \eqref{aacts} and \eqref{kappaacts}.

 \begin{corollary}   The subalgebras $\mathfrak{gl}_2(j,u,v)$ are independent of the choice of $c_j\in\widehat{L}$. In particular, 
$$\iota(\kappa c_j)=-\iota(c_j),$$  that is,  the elements $e_{j,u}$, $f_{j,v}$ are rescaled for different choices of $c_j$. 
     
 \end{corollary}

 \begin{proof} Recall from Equation (5) that $\kappa$ acts as $-1$ on $\R\{L\}$. The equality $\iota(\kappa c_j)=-\iota(c_j)$ follows. Hence the construction of the $\mathfrak{gl}_2(j,u,v)$ is independent of the choice of  $c_j$.
     \end{proof}

We make the following remarks:
    
  \begin{itemize}
      \item 
  Though the $\mathfrak{gl}_2(j,u,v)$ subalgebras corresponding to the imaginary simple roots are all isomorphic as $j>0$ varies, once the basis $\{h_1, h_2\}$ for $\mathfrak{h}$ is fixed as in Theorem \ref{main}, the $\mathfrak{gl}_2(j,u,v)$ subalgebras corresponding to distinct choices of $j$
    can be distinguished by the adjoint action of $h_2$.  In particular, we have  $[ h_2, e_{j,u} ] = j e_{j,u}$ and similarly $[ h_2, f_{j,v} ] = -j f_{j,v}.$ Note also that for $j=-1$ in Theorem~\ref{main}, we obtain the $\mathfrak{gl}_2$ subalgebra $\mathfrak{gl}_2(-1)$ of $\mathfrak{m}$ corresponding to the unique real simple root, with generators that are linear combinations of the generators in Lemma \ref{JLW_lemma}.

       \item Recall by Lemma \ref{existprim1} that given $u, v\in P_{j+1}^\natural,$
       \begin{equation*}
        u_{2j+1}v=(-1)^j(u,v)_{V^\natural}{\bf 1},
       \end{equation*}
       where $(\cdot, \cdot)_{V^\natural}$ is the bilinear form given by Proposition \ref{formdef}. Therefore, the hypothesis $u_{2j+1}v=\bf{1}$ of Theorem \ref{main} is equivalent to $(u,v)_{V^\natural}=(-1)^j.$

     \end{itemize}
    \begin{proposition}
        \label{corA}
       Let $j\in \{-1, 1, 2, \cdots\}$ and  let $0\neq u\in P_{j+1}^\natural$.  Then there exists $v\in P_{j+1}^\natural$ such that $u$ and $v$ satisfy the hypothesis of Theorem \ref{main}. That is, such that $u_{2j+1}v={\bf 1}.$ In particular, this implies that $\{e_{j, u}, f_{j, v}, h_1, h_2\}$ (as defined in Theorem \ref{main}) span a $\mathfrak{gl}_2$ subalgebra.
    \end{proposition}
    \begin{proof}
Since the bilinear form given by Proposition \ref{formdef} can be normalized to be positive definite, $(u,u)\ne 0$. Therefore, the pair $u, v$ where $v=(-1)^j\frac{u}{(u,u)}$ satisfies $u_{2j+1}v={\bf 1}.$
    \end{proof}
    \begin{corollary}\label{primary}
  Let $j\in \{-1, 1, 2, \cdots\}.$ Then there exists a pair of primary vectors $u,v\in P^\natural_{j+1}$ that satisfy the hypothesis of Theorem \ref{main}.
    \end{corollary}
  \begin{proof}
  By Theorem \ref{existprim}, there exists a nonzero primary vector $u$ of weight $j+1$. The result then follows from Proposition~\ref{corA}.
  \end{proof}

\section{The $\mathbb{M}$-action on $\mathfrak{gl}_2$ subalgebras}\label{ActionOfM}

 In this section, we show that the \cite{FLM} action of $\mathbb{M}$ on $V^\natural$, restricted to the subspace $P^\natural_{j+1}$ of primary vectors in $V^\natural_{j+1}$, induces an $\M$-action on the set of $\mathfrak{gl}_2(j,u,v)$ subalgebras corresponding to a fixed imaginary simple root.

\begin{theorem}\label{MactGj}Fix $j\in \{-1, 1, 2, \cdots\}$. Let $\mathfrak{gl}_2(j,u,v)$ be as defined in Theorem \ref{main}.   The \cite{FLM} action of $\M$ on $V^\natural$ induces an action on the following set of $\mathfrak{gl}_2$ subalgebras
\begin{equation*}
    \mathcal{G}_j=\left\{\mathfrak{gl}_2(j,u,v)\mid\ 
u,v\in P^\natural_{j+1},\ u_{2j+1}v={\bf 1}\right\}
\end{equation*}
defined by
$$g\cdot\mathfrak{gl}_2(j,u,v):=\mathfrak{gl}_2(j,g\cdot u,g\cdot v)$$
for $g\in\M$. In particular, $\mathcal{G}_{-1}$ contains the single subalgebra
$\mathfrak{gl}_2(-1)$ with trivial $\M$-action.
\end{theorem}
{\begin{proof}Let $u,v\in P_{j+1}^\natural$ such that $u_{2j+1}v={\bf 1}.$ Then $\mathfrak{gl}_2(j, u, v)\in \mathcal{G}_j$ is generated by
\begin{align*} e_{j,u}&=u\otimes \iota(c_j)+R,&
f_{j,v}&=v\otimes \iota(c_j^{-1})+R,&\\
h_1&={\bf 1}\otimes t_1(-1)\iota(1)+R,&
h_2&={\bf 1}\otimes t_2(-1)\iota(1)+R.&
\end{align*}
Recall that the action of $\M$ on $V^\natural$ induces an action of $\M$ on $V=V^\natural\otimes V_{1,1}$, where $\M$ acts trivially on $V_{1,1}$. Restricting this action to the subspace $P^\natural_{j+1}$ of primary vectors in $V^\natural_{j+1}$, we get an induced  action of $\M$ on $\mathfrak{m}$ (\cite{JLW}, \cite{BoInvent}). 

In particular, 
given any $g\in \M,$ we have 
\begin{align} g\cdot e_{j,u}&=g\cdot u\otimes \iota(c_j)+R,&
g\cdot f_{j,v}&=g\cdot v\otimes \iota(c_j^{-1})+R,&\\
g\cdot h_1&={\bf 1}\otimes t_1(-1)\iota(1)+R=h_1,&
g\cdot h_2&={\bf 1}\otimes t_2(-1)\iota(1)+R=h_2,&\label{hfixed}
\end{align} 
where in line \eqref{hfixed} we have used that the action of $\M$ on $V^\natural$ is by vertex operator algebra automorphisms and vertex operator algebra automorphisms fix the vacuum vector (\cite{FLM}, (8.10.22)).
 Since $P^\natural_{j+1}$ is an $\M$-module (Corollary~\ref{proper}), we have $g\cdot u$ and $g\cdot v\in P_{j+1}^\natural.$ Again using that the action of $\M$ on $V^\natural$ is by vertex operator algebra automorphisms,
$$(g\cdot u)_{2j+1}(g\cdot v)=g\cdot (u_{2j+1}v)=g\cdot {\bf 1}={\bf 1},$$
 (\cite{FLM}, (8.10.20), (8.10.22)).
 Thus, $g\cdot e_{j,u},$ $g\cdot f_{j,v}$, $g \cdot h_1$, and $g\cdot h_2$ are of the form given in Theorem \ref{main} and therefore generate $\mathfrak{gl}_2(j, g\cdot u, g\cdot v)\in \mathcal{G}_j$. \end{proof}}

\begin{remark}Since the action of $\M$ on $\mathfrak{m}$ is by Lie algebra automorphisms (\cite{BoInvent}, \cite{JLW}), it is obvious that $g\cdot e_{j,u},$ $g\cdot f_{j,v}, g\cdot h_1,$ and $g\cdot h_2$ in the above proof generate a $\mathfrak{gl}_2$ subalgebra of $\mathfrak{m}.$ Theorem \ref{MactGj} shows that the $\M$ action on $\mathfrak{m}$ induces an action of $\M$ on the set $\mathcal{G}_j$.  
\end{remark}

{\begin{lemma}\label{moduleE}For fixed $j>0$ and fixed $c_j\in \widehat{L}$ such that $\bar{c_j}=(1,j)$, 
\begin{align*}\label{spaceprim}
   E_j&:= \left\{e_{j,u}=u\otimes \iota({c_j})+R\mid u\in P^\natural_{j+1}\right\}\subseteq \mathfrak{m}_{(1,j)}\quad\text{ and }\\
   F_j&:= \left\{f_{j,u}=u\otimes \iota({c_j}^{-1})+R\mid u\in P^\natural_{j+1}\right\}\subseteq \mathfrak{m}_{(-1,-j)}
    \end{align*}
\end{lemma}are $\M$-submodules of $\mathfrak{m}_{(1,j)}$ and $\mathfrak{m}_{(-1,-j)}$, respectively.
\begin{proof}Given any $g\in \M$ and $e_{j,u}=u\otimes \iota({c_j})+R\in E_j$ we have
\begin{equation*}
g\cdot e_{j,u}=g\cdot(u\otimes \iota({c_j})+R)=(g \cdot u)\otimes \iota(c_j)+R.
\end{equation*} By Corollary \ref{proper}, $P_{j+1}^\natural$ is an $\M$-submodule of $V_{j+1}^\natural$, so $g\cdot u\in P_{j+1}^\natural$ and hence $g\cdot e_{j,u}\in E_j$. Therefore, $E_j$ is an $\M$-module. A similar proof shows that $F_j$ is an $\M$-module.
\end{proof}}

\begin{theorem}\label{isomorphic}The $\M$-modules $P_{j+1}^\natural$, $E_j$, and $F_j$ are isomorphic.
\end{theorem}
\begin{proof}We will prove that $E_j$ is isomorphic to $P_{j+1}^\natural$. A similar proof shows that $F_j$ is isomorphic to $P_{j+1}^\natural.$
Let \[\Psi: P_{j+1}^\natural\rightarrow E_j\] be given by 
\[
\Psi(u)=u\otimes \iota(c_j)+R.
\] 
If $u\in P_{j+1}^\natural$ and $u\ne 0,$ then Proposition~\ref{corA} implies that $u\otimes \iota(c_j)\notin R.$ Therefore, $\Psi$ is injective. The map $\Psi$ is also clearly surjective. Furthermore, for all $g\in \mathbb{M}$ and for all $u\in P_{j+1}^\natural,$ \begin{equation*}
\Psi(g\cdot u)=(g\cdot u)\otimes \iota(c_j)+R=g\cdot(u\otimes \iota(c_j)+R)=g\cdot\Psi(u).\end{equation*} Therefore, $\Psi$ is also an $\M$-module homomorphism and thus is an isomorphism.
\end{proof}
}
Theorem \ref{isomorphic} implies the following.

\begin{corollary}
For $j>1$, the $\M$-modules $E_j$ and $F_j$ are proper submodules of $\mathfrak{m}_{(1,j)}$ and $\mathfrak{m}_{(-1,-j)}$, respectively.
\end{corollary}
\begin{proof}
By The No-ghost Theorem,  $\mathfrak{m}_{(1,j)}$ and $\mathfrak{m}_{(-1,-j)}$ are isomorphic as $\M$-modules to $V_{j+1}^\natural$. Since  $P^\natural_{j}$ is a proper $\M$-submodule of $V^\natural_{j}$ for $j>1$, 
the result then follows by Corollary \ref{proper}.
\end{proof}

\begin{theorem}\label{non-trivialactionR}Suppose that $j\in \{-1,1,2,\dots\}$ is such that $\dim P^\natural_{j+1}$ is strictly greater than the multiplicity of the trivial $\M$-module in the  decomposition of $V^\natural_{j+1}$ as a direct sum of irreducible $\M$-modules.
Then the action of $\M$ on $\mathcal{G}_j $ is non-trivial.
\end{theorem}

\begin{proof}By Theorem \ref{isomorphic}, $E_j$ and  $P_{j+1}^\natural$ are isomorphic as $\M$-modules.
By Proposition~\ref{corA}, each $e_{j,u}\in E_j$ is a generator of $\mathfrak{gl}_2(j,u,v)\in \mathcal{G}_j$. Recall that $e_{j,u}\in \mathfrak{m}_{(1,j)}$, and by the No-ghost Theorem,  $\mathfrak{m}_{(1,j)}\cong V_{j+1}^\natural$ as $\M$-modules. 
Thus, if each $e_{j,u}$ was fixed by the $\M$-action, there would have to be at least $\dim P^\natural_{j+1}$ trivial $\M$-submodules in the  decomposition of $V_{j+1}^\natural$ into irreducible $\M$-modules. Therefore, if dim $P_{j+1}$ is greater than the multiplicity of the trivial $\M$-module in the  decomposition of $V^\natural_{j+1}$ as a direct sum of irreducible $\M$-modules, the $\M$-action on $\mathcal{G}_j$ cannot be trivial. \end{proof}

\begin{conjecture}\label{ConjAction}
    The action of $\M$ on $\mathcal{G}_j$ is non-trivial for all $j>0.$
\end{conjecture}

    We have verified Conjecture~\ref{ConjAction} for $0<j<100$ as follows. By Theorem \ref{non-trivialactionR}, it is sufficient to prove that  $\dim P_{j+1}^\natural$ is strictly larger than the multiplicity of the trivial representation in the  decomposition of $V^\natural_{j+1}$ into irreducible $\M$-modules. In Appendix \ref{sectiongen}, we show how to compute the dimension of $P^\natural_{j+1}$ using modular forms (Proposition \ref{HaradaLang}) and in Appendix \ref{NonTriv}, we show how to compute the multiplicity  of the trivial representation in $V^\natural_{j+1}$ using recursion relations. We have used PARI/GP \cite{PARI2} to verify that the dimension of $P^\natural_{j+1}$ is strictly larger than the multiplicity  of the trivial representation for $0<j<100.$  This proves the following.
\begin{proposition}\label{NonTrivial}
    
      The action of $\M$ on $\mathcal{G}_j$ is non-trivial for  $0<j<100$. %
      \end{proposition}

\appendix
\section{Generating series}\label{sectiongen}
In this section, we recall the generating series for the dimensions of the subspaces $P_j^\natural$ of primary vectors of homogeneous weight $j$ in $V^\natural$ (Proposition~\ref{HaradaLang}). We use this series to prove that given any nonnegative integer $j$ not equal to $1$, there exist primary vectors of weight $j$ in $V^\natural$ (Theorem \ref{existprim}).

Let $q=e^{2\pi i \tau}$, {where $\tau \in \C$ with Im$(\tau)>0$.} %
Recall that the modular {$J$}-function can be expressed as
\begin{equation}\label{jfunsigma}J(q)=\frac{(1+240\sum_{k=1}^\infty\sigma_3(k)q^k)^3}{q\prod_{k=1}^\infty(1-q^k)^{24}}-744,
\end{equation}where $\sigma_3(k):=\displaystyle\sum_{d|k}d^3$ is the sum of cubes of the positive integers that divide~$k$. Define
\begin{equation}\label{etafun}
\eta(q)=q^{\frac{1}{24}}\prod_{j=1}^\infty(1-q^j).
\end{equation}
It will also be helpful to recall the generating series for the partitions of $n$,
\begin{equation}\label{partfun}
\sum_{j=0}^\infty p(j)q^j={\prod_{j=1}^\infty(1-q^j)^{-1}},
\end{equation}where $p(j)$ denotes the number of partitions of $j\in \N$. %
We recall the Euler identity,
\begin{proposition}\label{eulerfun}
\begin{equation*}q^{-\frac{1}{24}}\eta(q)=\prod_{j=1}^\infty(1-q^j)=1+\sum_{j=1}^\infty (-1)^j(q^{\frac{j(3j+1)}{2}}+q^{\frac{j(3j-1)}{2}}).
\end{equation*}
\end{proposition}

{The following result can be obtained using the results of \cite{HL}. 
It is stated explicitly in \cite{DH}, with a minor typo. For the reader's convenience, we include a proof here.}

\begin{proposition}\cite{HL}\label{HaradaLang} (see also \cite{DH})
Let $P_{j}^\natural$ be the subspace of weight $j$ primary vectors in $V_j^\natural$. Then 
\begin{equation*}
    \sum^{\infty}_{j=0}\dim P_{j}^\natural \,q^{j-1}=q^{-\frac{1}{24}}J(q)\eta(q)+1.
\end{equation*}
\end{proposition}
\begin{proof}
We denote by $W_1, W_2,\cdots, W_{194}$ the irreducible representations of $\M$, with corresponding irreducible characters $\chi_1,\chi_2,\cdots, \chi_{194}.$ Then, since $V^\natural$ is an $\mathbb{M}$-module, we can write it as a sum of these irreducible modules. In fact, since the action of the Virasoro algebra commutes with the action of $\M$, we can decompose $V^\natural_j$, for each $j$, in the following way:
\begin{equation}\label{mult}
V_j^\natural=\bigoplus^{194}_{k=1} W_k^{\mult_{k}(j)},
\end{equation}
where $\mult_{k}(j)$ is the
multiplicity
of $W_k$ in $ V_j^\natural.$ Following the notation in \cite{HL}, we write $s^k_j$  for the number of linearly independent primary vectors in $V^{\natural,k}_{j}:=W_k^{ \mult_{k}(j)}$,
where $j\geq 0$ and $1\leq k \leq 194$.
 For $j\not=1,$ the dimension of $P_{j}^\natural$ is the sum over $k$ of the $s^k_j$: 
$$s_j:=\dim P_{j}^\natural=\sum^{194}_{k=1}s^k_j.$$ Following \cite{HL}, we set $s^1_1=-1$ for convenience. We define $$G^k(q)=\sum_{j\geq 0}s^k_j q^j,$$ and by \cite{HL}, we have: 
\begin{equation*}
    q^{-\frac{23}{24}}G^k(q)=(\dim{W_k})t_k(q)\eta(q),
\end{equation*} where, 
\begin{equation*}
    t_k(q)=\frac{1}{|\M|}\sum_{g\in\M}\chi_k(g)T^\natural_g(q).
\end{equation*} Here, $T^\natural_g(q)$ is the McKay-Thompson series for $V^\natural$ corresponding to the element $g\in \M.$ In particular, for $g=1$, the identity element in $\M,$ we have $T^{\natural}_{1}(q)=J(q).$

Combining the above two equations, we get:
\begin{equation*}
    q^{-\frac{23}{24}}G^k(q)=\frac{1}{|\M|}\left(\sum_{g\in\M}\dim{W_k}\chi_k(g)\right)T^\natural_g(q)\eta(q)
\end{equation*} 

Since we are interested in $s_j=\sum^{194}_{k=1}s^k_j$, we sum the above equation over~$k$:
\begin{align*}
\begin{split}
    q^{-\frac{23}{24}}\sum^{194}_{k=1} G^k(q)&=\frac{1}{|\M|}\sum_{g\in\M}\left(\sum^{194}_{k=1}\dim{W_k}\chi_k(g)\right)T^\natural_g(q)\eta(q)\\
    q^{-\frac{23}{24}}\sum^{194}_{k=1} \sum_{j\geq 0}s^k_j q^j
    &=\frac{1}{|\M|}\sum_{g\in\M}\chi_{\reg}(g)T^\natural_g(q)\eta(q)\\
    q^{-\frac{23}{24}}\sum_{j\geq 0}s_j q^j&=\frac{1}{|\M|}|\M|T^\natural_1(q)\eta(q)=J(q)\eta(q)
\end{split}
\end{align*} 
where $\chi_{\reg}$ is the character of the \textit{regular representation} of the Monster and we have used:
$$\chi_{\reg}(g)=\sum^{194}_{k=1}(\dim{W_k})\chi_k(g)\quad \text{   and   }\quad\chi_{\reg}(g)=\begin{cases}0 &\text{ if } g \not=\mathrm{id}\\
    |\M| &\text{ if } g =\mathrm{id}\end{cases}.$$

    Thus, adjusting for the fact that we set $s_1=-1$ whereas $\dim  P^\natural_1=0,$ we get: 
    \begin{align*}
             \sum_{j\geq 0}\dim  \, P^\natural_j q^j&= q^{\frac{23}{24}}J(q)\eta(q)+q
    \end{align*}
    which proves our desired equality. \end{proof}

\begin{theorem}\label{existprim} Let $j$ be a nonnegative integer not equal to $1$. Then there exist nonzero primary vectors in~$V_j^\natural$.
\end{theorem}
\begin{proof}

Applying Proposition \ref{HaradaLang}, Equations \eqref{etafun} and \eqref{jfunsigma}, then simplifying and applying Equation \eqref{partfun}, and lastly, applying Proposition \ref{eulerfun}, we have
\begin{eqnarray}\label{gen_prim_eta}
\sum_{j=0}^\infty \dim P_j^\natural q^{j-1}&=&\frac{1}{q}\left(1+240\sum_{j=1}^\infty \sigma_3(j)q^j\right)^3\left(\sum_{j=0}^\infty p(j)q^j\right)^{23}\\&&-744\left(1+\sum_{j=1}^\infty(-1)^j\left(q^{\frac{j(3j+1)}{2}}+q^{\frac{j(3j-1)}{2}}\right)\right)+1.\nonumber
\end{eqnarray}
By expanding the right hand side of Equation \eqref{gen_prim_eta} in powers of $q$, one can show that the coefficient of $q^{j-1}$ is strictly positive for $j=0$ and $j>1$, and all other coefficients are zero. Therefore, for all $j$ such that  $j=0$ or $j>1$, there exist primary vectors in $V_j^\natural$.
\section{Non-triviality of the $\M$-action}\label{NonTriv}
We denote by $W_1, W_2,\cdots, W_{194}$ the irreducible representations of $\M$, with corresponding irreducible characters $\chi_1,\chi_2,\cdots, \chi_{194}.$ Then, for $j\in\{-1,1,2,\dots\},$ we have the  decomposition \eqref{mult}:
\begin{equation*}
V_{j+1}^\natural=\bigoplus^{194}_{k=1} W_k^{ \mult_{k}(j+1)},
\end{equation*}
where $\mult_{k}(j+1)$ is the
multiplicity
of $W_k$ in $ V_{j+1}^\natural.$ As remarked at the end of Section \ref{ActionOfM}, the truth of Conjecture \ref{ConjAction} would follow from showing that for all $j\in\{-1,1,2,\dots\},$ $s_{j+1}:=\dim P^\natural_{j+1}$ is strictly larger than $\mult_{1}(j+1),$ the multiplicity of the trivial representation $W_1$ in $ V_{j+1}^\natural$.%

By character orthogonality, we have
\begin{equation*}
\mult_{k}(j+1)=\frac{1}{|\M|}\sum_{g\in \M} \overline{\chi_{k}(g)}\tr{(g|V^\natural_{j+1})}.
\end{equation*}

We let $C(g,j)$ and $C(g^2,j)$ denote the traces $\tr{(g|V^\natural_{j+1})}$ and $\tr{(g^2|V^\natural_{j+1})}$  respectively. Then we have the well-known recursion relations (see \cite{BoInvent}),
\begin{align*}
    C(g, 4j) =& ~
C(g, 2j+1) + \frac{C(g, j)^2 - C(g^2, j)}{2} + \sum_{i=1}^{j-1} C(g, i) \cdot C(g, 2j-i),\\
C(g,4j+1)=&~
C(g, 2j+3) - C(g, 2) \cdot C(g, 2j) + \frac{C(g, 2j)^2 + C(g^2, 2j)}{2} + \\&\frac{C(g, j+1)^2 - C(g^2, j+1)}{2} + \sum_{i=1}^{j} C(g, i) \cdot C(g, 2j-i+2) + \\&\sum_{i=1}^{j-1} C(g^2, i) \cdot C(g, 4j-4i) + \sum_{i=1}^{2j-1} (-1)^i \cdot C(g, i) \cdot C(g, 4j-i),\\
C(g,4j+2)=&~
C(g, 2j+2) + \sum_{i=1}^{j} C(g, i) \cdot C(g, 2j-i+1), \\
C(g,4j+3)=&~
C(g, 2j+4) - C(g, 2) \cdot C(g, 2j+1) - \frac{C(g, 2j+1)^2 - C(g^2, 2j+1)}{2} \\&+
 \sum_{i=1}^{j+1} C(g, i) \cdot C(g, 2j-i+3) + \sum_{i=1}^{j} C(g^2, i) \cdot C(g, 4j-4i+2) +\\& \sum_{i=1}^{2j} (-1)^i \cdot C(g, i) \cdot C(g, 4j-i+2). \end{align*}

In particular, for $j=4$ or $j > 5$, the
coefficient $C(g,j)$ is determined by the coefficients $C(g,i)$ and $C(g^2,i)$ for $1\leq i < j.$ We use the in-built modular forms repository in PARI/GP \cite{PARI2} to compute the coefficients of:
\begin{equation*}
    \sum^{\infty}_{j=-1}\dim P_{j+1}^\natural \,q^{j}=q^{-\frac{1}{24}}J(q)\eta(q)+1,
\end{equation*}

since $J(q)\eta(q)$ is a weight $\frac 12$ modular form. This allows us to compare the values of \begin{equation*}
    \mult_1(j+1)=\frac{1}{|\M|}\sum_{g\in \M} C(g,j)
\end{equation*} to the dimensions of $P^\natural_{j+1}$ and thus verify Conjecture \ref{ConjAction} for $1<j<100.$ We give the first few values in the table below.

{\small{\begin{table}[htp]
\begin{tabular}{|c||cc|}
\hline
$j$ & $\dim P^\natural_{j+1}$                                                       & {$\mult_1(j+1)$} \\ \hline
1                        & 196883                                                                      & 1                                  \\
2                        & 21296876                                                                    & 1                                  \\
3                        & 842609326                                                                   & 2                                  \\
4                        & 19360062527                                                                 & 2                                  \\
5                        & 312092484374                                                                & 4                                  \\
6                        & 3898575000125                                                               & 4                                  \\
7                        & 40071789624999                                                              & 7                                  \\
8                        & 352582733780823                                                             & 8                                  \\
9                        & 2730312616406501                                                            & 12                                 \\
10                       & 18989796260093750                                                           & 14                                 \\
11                       & 120472350229297625                                                          & 22                                 \\
12                       & 705579405073375001                                                          & 25                                 \\
13                       & 3851890223522607078                                                         & 36                                 \\
14                       & 19754724655128969898                                                        & 44                                 \\
15                       & 95796047847905125001                                                        & 61                                 \\
16                       & 441630416897735940875                                                       & 74                                 \\
17                       & 1944474605043319578125                                                      & 102                                \\
18                       & 8208966820642976271948                                                      & 124                                \\
19                       & 33342403696070463426523                                                     & 167                                \\
20                       & 130682291183967925390625                                                    & 206                                \\
21                       & 495541230687128562902875                                                    & 271                                \\
22                       & 1822158321664159999078124                                                   & 335                                \\
23                       & 6510652458052884364952274                                                   & 440                                \\
24                       & 22645881565834844801406026                                                  & 542                                \\
25                       & 76805694478383734573046875                                                  & 701                                \\
26                       & 254378447193404062648279992                                                 & 870                                \\
27                       & 823820250669449124864265625                                                 & 1115                               \\
28                       & 2612037978193398885792057928                                                & 1381                               \\
29                       & 8117168463824355581684218453                                                & 1762                               \\
30                       & 24748559924646442300596578125                                               & 2180                               \\
31                       & 74100585128385505089520426375                                               & 2763                               \\
32                       & 218068784814065333189473046875                                              & 3422                               \\
33                       & 631263434817949765287221989496                                              & 4310                               \\
34                       & 1798839455374997664745734472049                                             & 5333                               \\
35                       & 5049345338644493766280585734376                                             & 6697                               \\
36                       & 13970568011333638480233896790625                                            & 8272                               \\
37                       & 38122902172895468426986907453125                                            & 10342                              \\
38                       & 102657396484068599392862170371503                                           & 12773                              \\
39                       & 272929768681646094007878106129219                                           & 15913                              \\
40                       & 716766590714096093408391800296876                                           & 19624                              \\
41                       & 1860234399965047844989826549991625                                          & 24386                              \\
42                       & 4773156795988402310139116350828125                                          & 30034                              \\
43                       & 12113398911563006366044489650277199                                         & 37219                              \\
                  \hline
\end{tabular}
\end{table}
}}

\newpage

\end{proof}

\bibliographystyle{amsalpha}
\newcommand{\etalchar}[1]{$^{#1}$}
\providecommand{\bysame}{\leavevmode\hbox to3em{\hrulefill}\thinspace}
\providecommand{\MR}{\relax\ifhmode\unskip\space\fi MR }
\providecommand{\MRhref}[2]{%
  \href{http://www.ams.org/mathscinet-getitem?mr=#1}{#2}
}
\providecommand{\href}[2]{#2}

\end{document}